\documentclass{article}
\usepackage{iclr2027_preprint, times}

\usepackage{url}
\usepackage{amsmath,amssymb,amsthm}
\usepackage{booktabs}
\usepackage{graphicx}
\usepackage{algorithm}
\usepackage{algorithmic}
\usepackage{enumitem}
\usepackage{xcolor}
\usepackage{hyperref}
\usepackage[capitalize]{cleveref}

\theoremstyle{plain}
\newtheorem{theorem}{Theorem}[section]
\newtheorem{proposition}[theorem]{Proposition}
\newtheorem{lemma}[theorem]{Lemma}

\theoremstyle{definition}
\newtheorem{definition}[theorem]{Definition}
\newtheorem{assumption}[theorem]{Assumption}
\newtheorem{remark}[theorem]{Remark}
\newtheorem{example}[theorem]{Example}

\newcommand{\R}{\mathbb{R}}
\newcommand{\Cx}{\mathbb{C}}
\newcommand{\E}{\mathbb{E}}
\newcommand{\Prob}{\mathbb{P}}
\newcommand{\Var}{\mathrm{Var}}
\newcommand{\KL}{\mathrm{KL}}
\newcommand{\TV}{\mathrm{TV}}

\newcommand{\conv}{\mathrm{conv}}
\newcommand{\Unif}{\mathrm{Unif}}
\newcommand{\Bern}{\mathrm{Bern}}
\newcommand{\sphere}{\mathbb{S}^{d-1}}
\newcommand{\ball}{\mathbb{B}_d}
\newcommand{\Dtil}{\widetilde{\Delta}}
\newcommand{\Dhat}{\widehat{\Delta}}
\newcommand{\gap}{\Delta}
\newcommand{\Df}{\Delta_f}
\newcommand{\Clink}{C_{\mathrm{link}}}
\newcommand{\cs}{c_\sigma}
\newcommand{\gnorm}[2]{\big\|\bar\partial_{#1} #2\big\|}
\newcommand{\lens}{\Lambda}
\newcommand{\vhat}{\hat v}
\newcommand{\eps}{\varepsilon}

\title{Nonsmooth Nonconvex Optimization from Noisy Comparisons:\\ Unbiased Gap Estimation and the Price of Preference Feedback}

\author{El Mahdi Chayti\\ Machine Learning and Optimization Laboratory (MLO)\\ EPFL \And Taha El Bakkali\\ College of Computing\\ UM6P \And Omar Saadi\\ College of Computing\\ UM6P}
\iclrfinalcopy
\begin{document}
\maketitle

\begin{abstract}
We study the minimization of $L$-Lipschitz, possibly nonsmooth and nonconvex $f:\R^d\to\R$ from noisy pairwise comparisons: comparing $x$ and $y$ returns a bit equal to $1$ with probability $\sigma((f(y)-f(x))/\tau)$, for a known link $\sigma$ and temperature $\tau$, and the goal is a $(\delta,\eps)$-Goldstein stationary point. We give an exactly unbiased estimator of the gap $f(y)-f(x)$ whenever $\sigma^{-1}$ is analytic on a lens in the complex plane, as for the logistic, probit and cauchit links: compare the pair until the outcome changes and return $\pm\tau$ times a partial sum of the expansion coefficients of the inverse link, the harmonic number of the run length for the logistic link. It uses at most $1+1/(1-\alpha)$ comparisons in expectation, $\alpha$ the largest queried success probability, and has second moment of order the gap squared plus $\tau^2$. With randomized smoothing and the online-to-nonconvex conversion this gives $O\big(d\,(1+d\,\Clink/\kappa^2)\,L^2\Df/(\delta\eps^3)\big)$ comparisons when $\Df\ge L\delta$, where $\kappa$ is the ratio of the largest queried gap to $\tau$, which the smoothing radius caps at a constant; relative to the optimal rate $O(dL^2\Df/(\delta\eps^3))$ for noisy value oracles, comparisons cost a factor $d/\kappa^2$. We prove that any comparison-based method needs $\Omega(d^2\tau^2/(\eps^2\delta^2))$ comparisons on a smooth convex subclass once $d\gtrsim(\eps\delta/\tau)^2$: in the comparison-noise-dominated regime the bounds agree in $d$ and $\tau$, and a factor $\Df/(\delta\eps)$ remains open. Experiments illustrate the estimator's cost and variance and the $d^2$ growth of the gradient-estimator second moment.
\end{abstract}

\section{Introduction}
\label{sec:intro}

Comparison-based optimization models settings in which an algorithm cannot observe values of an objective $f$ but can ask which of two candidates is better and receive a noisy answer, as in dueling bandits, human-in-the-loop tuning, and preference signals in RLHF-style pipelines \citep{saha2021dueling,tang2024zeroth,chen2025compo}. For smooth objectives, convergence guarantees under noisy comparisons are available in the convex \citep{saha2021dueling,saha2025dueling} and the nonconvex \citep{bergou2020stochastic,elbakkali2025almost,bakkali2026noisy} regimes. The nonsmooth nonconvex regime, where gradients need not exist and the standard solution concept is $(\delta,\eps)$-Goldstein stationarity \citep{goldstein1977optimization,zhang2020complexity}, has, to our knowledge, not been treated with comparison feedback prior to this work. With noisy \emph{value} feedback it is well understood: randomized smoothing and two-point differences give $O(d^{3/2}\delta^{-1}\eps^{-4})$ evaluations \citep{lin2022gradient}, improved to $O(d^{3/2}\delta^{-1}\eps^{-3})$ by variance reduction \citep{chen2023faster} and to the dimension-optimal $O(d\,\delta^{-1}\eps^{-3})$ by \citet{kornowski2023algorithm} through the online-to-nonconvex conversion of \citet{cutkosky2023optimal}. Those methods need the magnitude of value differences, which comparisons do not reveal.

We work with the standard oracle of preference-based optimization: for $x,y\in\R^d$ the oracle returns $Y\in\{0,1\}$ with $\Prob(Y=1)=\sigma\big((f(y)-f(x))/\tau\big)$, where the link $\sigma:\R\to(0,1)$ is known, strictly increasing and symmetric, and $\tau>0$ is a temperature. The central quantity throughout is
\[
\kappa \;:=\; \frac{B}{\tau}, \qquad B := \sup\text{ of the queried gaps } |f(y)-f(x)|,
\]
the gap-to-temperature ratio. Our queries are symmetric pairs $x\pm\rho u$ with $\|u\|=1$, so $B=2L\rho$. The smoothing radius $\rho$ is a parameter distinct from the Goldstein radius $\delta$: it must satisfy $\rho\le\delta/2$, and the main algorithm sets $\rho=\min\{\delta/2,\kappa^\star\tau/(2L)\}$ for a constant $\kappa^\star$, so that $\kappa=\min\{\kappa^\star,L\delta/\tau\}$.

\paragraph{Contributions.}
\begin{itemize}[leftmargin=*,itemsep=1pt,topsep=2pt]
\item \textbf{Unbiased gap estimation from comparisons (Section~\ref{sec:estimator}).} We expand the inverse link $g=\sigma^{-1}$ in the polynomials $b_m(p)=p^m-(1-p)^m$, whose natural estimators from one stream of comparisons vanish after the first change of outcome, so the whole series is computable from finitely many comparisons: $\Dtil$ compares the pair until the outcome changes and returns $\pm\tau$ times a partial sum of the coefficients. It is exactly unbiased, uses at most $1+1/(1-\alpha)$ comparisons in expectation, and has $\E[\Dtil^2]\le\Clink\tau^2$ with $\Clink$ depending only on the link and $\kappa$ ($13(1+\kappa)^2$ for the logistic link). The expansion exists when $g$ is analytic on a lens $\{|z|\le r\}\cap\{|z-1|\le r\}$ of the complex plane (logistic, probit and cauchit links), and analyticity near the real interval alone is not sufficient.
\item \textbf{Algorithms and rates (Section~\ref{sec:algorithms}).} The gradient estimator $G=\frac{d}{2\rho}\Dtil\, u$ is unbiased for $\nabla f_\rho$ with second moment at most $dL^2(c_1+d\,\Clink/\kappa^2)$. Plugged into the online-to-nonconvex conversion, it yields $(\delta,\eps)$-Goldstein stationarity with $O\big((1+\tfrac{1}{1-\alpha})\,d(c_1+d\Clink/\kappa^2)\,L^2\Df/(\delta\eps^3)\big)$ comparisons when $\Df\ge L\delta$ (Theorem~\ref{thm:main}), where the conversion's insensitivity to the smoothness of $f_\rho$ lets us choose $\rho$ to keep $\kappa\le\kappa^\star$, so the bound is $O(d^2L^2\Df/(\delta\eps^3))$ for every $\tau\le L\delta/\kappa^\star$; plugged into SGD on the smoothed function, where $\rho$ does enter through the smoothness $\sqrt dL/\rho$, it yields $O\big((1+\tfrac{1}{1-\alpha})\,d^{3/2}(c_1+d\Clink/\kappa^2)L^3\Df/(\delta\eps^4)\big)$ with $\kappa=2L\delta/\tau$ (Theorem~\ref{thm:sgd}).
\item \textbf{A lower bound in $d$ and $\tau$ (Section~\ref{sec:lower}).} For any link with $\cs<\infty$ (a shifted Fisher-information bound met by the logistic, probit and cauchit links), any algorithm that outputs a $(\delta,\eps)$-Goldstein stationary point in expectation for every convex, $L$-Lipschitz, $(\eps/\delta)$-smooth $f$ with $f(x_0)-\inf f=O(\eps\delta)$ needs $T=\Omega(d^2\tau^2/(\eps^2\delta^2))$ in high dimension (Theorem~\ref{thm:lower}); this agrees with the upper bound in $d$ and $\tau$ when $\kappa$ is bounded and leaves a factor $\Df/(\delta\eps)$ open. The proof is an Assouad argument: a comparison is a Bernoulli observation whose Fisher information about the minimizer has trace $O(\eps^2/\tau^2)$, and locating $d$ coordinates from total information $\mathcal I$ costs mean squared error $d^2/\mathcal I$.
\item \textbf{Experiments (Section~\ref{sec:experiments}).} Simulations illustrate the cost and variance of the estimator, the $d^2$ growth of $\E\|G\|^2$, the $d^2/T$ estimation rate on the hard family of the lower bound, and the behavior of both algorithms on a nonsmooth nonconvex test problem.
\end{itemize}

Table~\ref{tab:rates} summarizes the rates. The upper and lower bounds match in their dependence on $d$ and $\kappa$ in the regime $\tau\ge L\delta/\kappa^\star$, where $\kappa=L\delta/\tau$ and $\Clink$ is an absolute constant; the dependence on $\delta$ and $\eps$ of the lower bound is not tight, and we state the joint question as open in Section~\ref{sec:lower}.

\begin{table}[t]
\centering\small
\caption{Oracle complexity for a $(\delta,\eps)$-Goldstein stationary point of an $L$-Lipschitz $f$ with $f(x_0)-\inf f\le\Df$, assuming $\Df\ge L\delta$; constants omitted. For comparisons, $\kappa=\min\{\kappa^\star,L\delta/\tau\}$ (O2NC) or $2L\delta/\tau$ (SGD), $\Clink=\Clink(\sigma,\kappa)$, and $\E[N]\le1+1/(1-\alpha)$ with $\alpha=\sigma(\kappa)$ is the expected number of comparisons per gap estimate.}
\label{tab:rates}
\resizebox{\linewidth}{!}{\begin{tabular}{llll}
\toprule
Oracle & Method & Complexity & Reference \\
\midrule
values, exact or Lipschitz per sample & smoothing + SGD & $d^{3/2}L^3\Df/(\delta\eps^4)$ & \citet{lin2022gradient} \\
values, Lipschitz per sample & smoothing + O2NC & $d\,L^2\Df/(\delta\eps^3)$ & \citet{kornowski2023algorithm} \\
comparisons & smoothing + SGD & $\E[N]\,d^{3/2}\big(1+\frac{d\,\Clink}{\kappa^2}\big)L^3\Df/(\delta\eps^4)$ & Thm.~\ref{thm:sgd} \\
comparisons & smoothing + O2NC & $\E[N]\,d\,\big(1+\frac{d\,\Clink}{\kappa^2}\big)L^2\Df/(\delta\eps^3)$ & Thm.~\ref{thm:main} \\
comparisons & any & $\Omega\big(d^2\tau^2/(\eps^2\delta^2)\big)=\Omega\big(d^2L^2/(\kappa^2\eps^2)\big)$ & Thm.~\ref{thm:lower} \\
\bottomrule
\end{tabular}}
\end{table}

\section{Preliminaries}
\label{sec:prelim}

\paragraph{Notation.} $\|\cdot\|$ is the Euclidean norm, $\ball$ the closed unit ball and $\sphere$ the unit sphere in $\R^d$. Throughout, $f:\R^d\to\R$ is $L$-Lipschitz and bounded below, $x_0$ is the initial point and $\Df\ge f(x_0)-\inf f$.

\paragraph{Goldstein stationarity.} By Rademacher's theorem $f$ is differentiable almost everywhere, and its Clarke subdifferential $\partial f(x)$ is the convex hull of all limits of $\nabla f(x_n)$ along differentiable points $x_n\to x$ \citep{clarke1990optimization}. The Goldstein $\delta$-subdifferential is $\partial_\delta f(x):=\conv\big(\bigcup_{\|y-x\|\le\delta}\partial f(y)\big)$ and $\gnorm{\delta}{f(x)}:=\min\{\|g\|:g\in\partial_\delta f(x)\}$. A point $x$ is $(\delta,\eps)$-Goldstein stationary if $\gnorm{\delta}{f(x)}\le\eps$ \citep{goldstein1977optimization,zhang2020complexity}. If $f$ is $H$-smooth then $\|\nabla f(x)\|\le\gnorm{\delta}{f(x)}+H\delta$, since every element of $\partial_\delta f(x)$ is a convex combination of gradients at points within $\delta$ of $x$.

\paragraph{Randomized smoothing.} For $\rho>0$ let $f_\rho(x):=\E_{z\sim\Unif(\ball)}f(x+\rho z)$. Then $f_\rho$ is $L$-Lipschitz, differentiable, $c_2\sqrt d\,L/\rho$-smooth for an absolute constant $c_2$, and $\|f-f_\rho\|_\infty\le L\rho$ \citep{yousefian2012stochastic}. Two facts connect $f_\rho$ to Goldstein stationarity of $f$: $\nabla f_\rho(x)\in\partial_\rho f(x)$ \citep{lin2022gradient}, and more generally $\partial_\nu f_\rho(x)\subseteq\partial_{\rho+\nu}f(x)$ for all $\nu\ge0$ \citep[Lemma~4]{kornowski2023algorithm}. The gradient of $f_\rho$ admits the two-point representation
\begin{equation}
\label{eq:twopoint}
\nabla f_\rho(x)=\frac{d}{2\rho}\,\E_{u\sim\Unif(\sphere)}\big[\gap_u(x)\,u\big],\qquad \gap_u(x):=f(x+\rho u)-f(x-\rho u),
\end{equation}
and the directional gap concentrates: $\E_u[\gap_u(x)^2]\le c_1L^2\rho^2/d$ for an absolute constant $c_1$ \citep[Lemma~10]{shamir2017optimal}; see \citet[Lemma~7]{kornowski2023algorithm} for the explicit constant $c_1=64\sqrt{2\pi}$. By Lipschitz continuity, $|\gap_u(x)|\le 2L\rho$ for every $u$.

\paragraph{Comparison oracle.} A link is a strictly increasing $\sigma:\R\to(0,1)$ with $\sigma(-t)=1-\sigma(t)$; its inverse is $g:=\sigma^{-1}:(0,1)\to\R$, which satisfies $g(1-p)=-g(p)$. Given $\tau>0$, a query on the pair $(x,y)$ returns $Y\in\{0,1\}$ with
\[
\Prob(Y=1\mid x,y)=\sigma(\gap/\tau),\qquad \gap:=f(y)-f(x),
\]
independently of everything else; repeated queries on the same pair are i.i.d. Writing $p:=\sigma(\gap/\tau)$, the gap is $\gap=\tau g(p)$. All queries made by our algorithms are symmetric pairs $(x-\rho u,x+\rho u)$ with $u\in\sphere$, for which $|\gap|\le B:=2L\rho$. Hence every queried success probability lies in the \emph{operated range} $[1-\alpha,\alpha]$ with
\begin{equation}
\label{eq:alpha}
\kappa:=B/\tau,\qquad \alpha:=\sigma(\kappa)\in(1/2,1).
\end{equation}
The three links we use as running examples are the logistic $\sigma(t)=1/(1+e^{-t})$, the probit $\sigma(t)=\Phi(t)$ with $\Phi$ the standard normal distribution function, and the cauchit $\sigma(t)=\frac12+\frac1\pi\arctan t$.

\section{Unbiased gap estimation from comparisons}
\label{sec:estimator}

The task is to estimate $\gap=\tau g(p)$ without bias from Bernoulli$(p)$ samples. A function of $p$ is unbiasedly estimable from $m$ Bernoulli samples only if it is a polynomial of degree at most $m$, so an unbiased estimator of $g$ must use a random number of samples: we expand $g$ in polynomials tailored to comparisons and truncate at a random index.

\subsection{A comparison-friendly polynomial basis}

\begin{definition}[Bernoulli-product polynomials and prefix estimators]
For $m\ge1$ let $b_m(p):=p^m-(1-p)^m$. Given a stream $Y_1,Y_2,\dots$ of i.i.d.\ $\Bern(p)$ outcomes, let $A_m:=\prod_{i\le m}Y_i$ and $B_m:=\prod_{i\le m}(1-Y_i)$ be the indicators that the first $m$ outcomes are all ones, respectively all zeros.
\end{definition}
Then $\E[A_m-B_m]=b_m(p)$, and because $A_m$ and $B_n$ cannot both equal one,
\begin{equation}
\label{eq:prefixproduct}
(A_m-B_m)(A_n-B_n)=A_{m\vee n}+B_{m\vee n},\qquad \E\big[(A_m-B_m)(A_n-B_n)\big]=p^{m\vee n}+(1-p)^{m\vee n},
\end{equation}
where $m\vee n:=\max(m,n)$. Each $b_m$ is odd about $1/2$, as is $g$. Expanding $g$ in the $b_m$ requires more than analyticity near the real interval $[1-\alpha,\alpha]$: if $\sum_m|c_m|\alpha^m<\infty$ then $\sum_mc_mz^m$ and $\sum_mc_m(1-z)^m$ converge on the discs $\{|z|<\alpha\}$ and $\{|z-1|<\alpha\}$, so any such expansion extends $g$ analytically to the open lens $\{|z|<\alpha\}\cap\{|z-1|<\alpha\}$, which reaches height $\sqrt{\alpha^2-1/4}$ above $1/2$. The converse holds with a slightly larger lens.

\begin{assumption}[Lens-analytic inverse link]
\label{ass:lens}
For the operated range $[1-\alpha,\alpha]$ of \eqref{eq:alpha} there is $r\in(\sqrt\alpha,1)$ such that $g=\sigma^{-1}$ extends holomorphically to an open neighborhood of the closed lens $\lens_r:=\{z\in\Cx:|z|\le r,\ |z-1|\le r\}$.
\end{assumption}

\begin{proposition}[Bernoulli-product expansion]
\label{prop:expansion}
Under Assumption~\ref{ass:lens} there is a function $h$ holomorphic on $\{|z|<r\}$ such that $g(z)=h(z)-h(1-z)$ on the interior of $\lens_r$. Let $c_m:=h^{(m)}(0)/m!$ for $m\ge1$ and, for $r'\in(\sqrt\alpha,r)$, let $C_{r'}:=\max_{|z|=r'}|h(z)|$. Then for every $p\in[1-\alpha,\alpha]$,
\[
g(p)=\sum_{m\ge1}c_m\,b_m(p),\qquad |c_m|\le C_{r'}\,(r')^{-m},
\]
and the series converges absolutely and uniformly on $[1-\alpha,\alpha]$.
\end{proposition}
The proof (Appendix~\ref{app:lens}) splits Cauchy's formula on $\lens_r$ into its two boundary arcs, which give functions holomorphic on $\{|z|<r\}$ and on $\{|z-1|<r\}$ respectively, and uses the oddness of $g$ to merge them into one $h$. The decomposition is not unique (for instance $b_2=b_1$), which is harmless: any admissible $(c_m)$ can be used below, and $c_m$ is a Cauchy integral of $g$ over an arc.

\begin{example}
\label{ex:links}
\emph{Logistic:} $g(p)=\log\frac{p}{1-p}=h(p)-h(1-p)$ with $h(z)=-\log(1-z)$, so $c_m=1/m$ and Assumption~\ref{ass:lens} holds for every $r<1$, hence for every $\alpha<1$. \emph{Probit:} $\Phi$ is entire with nowhere vanishing derivative, and its only asymptotic values are $0$ and $1$, so $\Phi^{-1}$ continues analytically to every simply connected domain avoiding $0$ and $1$ (Appendix~\ref{app:lens}), in particular to a neighborhood of $\lens_r$ for every $r<1$. \emph{Cauchit:} $g(p)=\tan(\pi(p-1/2))$ is meromorphic with poles at the integers only. In all three cases the assumption holds for every $\alpha<1$, with constants $C_{r'}$ that grow as $\alpha\to1$. A link whose inverse has a pole at $\tfrac12\pm i\eta$ with $\eta<\sqrt{\alpha^2-1/4}$ (Appendix~\ref{app:lens} gives one) is strictly increasing and analytic near $[1-\alpha,\alpha]$, yet admits no expansion with $\sum_m|c_m|\alpha^m<\infty$. The step link (noiseless sign feedback) has no analytic inverse at all, consistent with the fact that magnitudes cannot be recovered from signs.
\end{example}

\subsection{The run-length estimator}

The products $A_m$ and $B_m$ vanish for every $m$ past the first change of outcome. Hence the series $\sum_mc_m(A_m-B_m)$ has finitely many nonzero terms, and it can be computed exactly.

\begin{definition}[Run-length estimator]
\label{def:estimator}
Let $(c_m)$ be as in Proposition~\ref{prop:expansion}. Query the pair $(x,y)$ repeatedly and stop at the first outcome that differs from $Y_1$; let $N:=\min\{n\ge2:Y_n\ne Y_1\}$ be the number of comparisons used, and $S_k:=\sum_{m=1}^{k}c_m$. Output
\[
\Dtil(x,y):=\tau\,(2Y_1-1)\,S_{N-1}=\tau\sum_{m\ge1}c_m\,(A_m-B_m).
\]
\end{definition}
The two expressions agree because $A_m-B_m=2Y_1-1$ for $m<N$ and $A_m=B_m=0$ for $m\ge N$. For the logistic link $S_k=H_k$ is the $k$th harmonic number: a long run of identical outcomes is evidence of a large gap, and $H_k\approx\log k$ matches the logit. We assume exact arithmetic and access to the real coefficients $c_m$: $1/m$ for the logistic link, and in general the Cauchy integrals of $g$ over the arcs of $\partial\lens_r$ (Appendix~\ref{app:lens}), computable on demand; a perturbation $\delta c_m$ biases $\Dtil$ by at most $2\tau\sum_m|\delta c_m|\alpha^m$.

\begin{theorem}[Run-length estimator]
\label{thm:estimator}
Let $p\in[1-\alpha,\alpha]$ and suppose $|c_m|\le C\rho_c^{\,m}$ for all $m$ with $\alpha\rho_c<1$ (Proposition~\ref{prop:expansion} gives $\rho_c=1/r'$ and $C=C_{r'}$).
\begin{enumerate}[leftmargin=*,itemsep=0pt,topsep=2pt]
\item[(i)] $\E[\Dtil]=\gap$.
\item[(ii)] $\E[N]=1+\frac{p}{1-p}+\frac{1-p}{p}\le1+\frac1{1-\alpha}$.
\item[(iii)] Let $\bar\rho:=\max\{1,\rho_c\}$ and $\theta:=\alpha\bar\rho^{\,2}$. If $\theta<1$, which holds for $\rho_c=1/r'$ with $r'>\sqrt\alpha$, then
\[
\E[\Dtil^2]\;\le\;\Clink\,\tau^2,\qquad \Clink:=\frac{2C^2\theta(1+\theta)}{(1-\theta)^3}.
\]
For the logistic link, $\E[\Dtil^2]=\tau^2\sum_{k\ge1}\big(p^k(1-p)+(1-p)^kp\big)H_k^2$ exactly, and $\E[\Dtil^2]\le13\,\tau^2(1+\kappa)^2$.
\item[(iv)] $\E|\Dtil|^q<\infty$ whenever $\alpha\bar\rho^{\,q}<1$ (all moments for the logistic link).
\end{enumerate}
\end{theorem}
The proof is in Appendix~\ref{app:estimator}: unbiasedness is Fubini's theorem applied to the series, the cost follows from $\Prob(N>n)=p^n+(1-p)^n$, and the moments from $|\Dtil|=\tau|S_{N-1}|$ with $N-1$ geometric and $|S_k|\le Ck\bar\rho^{\,k}$, or $S_k=H_k\le1+\log k$ for the logistic link. The bound $13\tau^2(1+\kappa)^2$ says that at the edge of the operated range, where $\gap^2=\kappa^2\tau^2$, the second moment is at most a constant times $\gap^2+\tau^2$: nearly deterministic comparisons cost comparisons ($\E[N]\approx e^\kappa$), not accuracy. The exact series agrees with simulations to three digits and the bound is loose by a factor $15$ to $20$ (Table~\ref{tab:e1}). The randomly truncated (Russian-roulette) series of \citet{mcleish2011general,rhee2015unbiased,lyne2015russian} is also unbiased, and $\Dtil$ is its conditional expectation given the comparison stream, so it has no larger variance when that estimator has a finite second moment. Cost goes the other way: nested Russian roulette may stop at the earlier of its truncation time and the first outcome change, so it is cheaper, and at the settings we tried its product of cost and second moment is smaller. What separates them is the tail (all moments finite for $\Dtil$ under the logistic link, the $q$th moment when $\alpha\bar\rho^{\,q}<1$ in general, against divergent high moments under geometric truncation) and the absence of a parameter to tune; Appendix~\ref{app:rr} has the details, and either estimator can be used in Section~\ref{sec:algorithms}.

\begin{remark}[Dependence on $\kappa$]
\label{rem:kappa}
Both $\E[N]$ and $\Clink$ are functions of $\alpha=\sigma(\kappa)$ alone. As $\kappa\to0$ (gaps small relative to the temperature) $\E[N]\to3$ and $\Clink$ stays bounded, while $\gap^2\le\kappa^2\tau^2$ becomes negligible against the estimator variance of order $\tau^2$. As $\kappa\to\infty$ (comparisons almost deterministic) $\E[N]\approx1/(1-\alpha)$, which is $e^{\kappa}$ for the logistic link, while $\E[\Dtil^2]/\gap^2$ stays bounded. Heuristically the cost growth is intrinsic: one comparison carries Fisher information $I_\sigma(\gap/\tau)/\tau^2$ about $\gap$, with $I_\sigma(t):=\sigma'(t)^2/(\sigma(t)(1-\sigma(t)))\approx e^{-|t|}$ for the logistic link, so estimating a gap at the edge of the operated range to accuracy $\tau$ needs of order $1/(1-\alpha)$ comparisons, which is what $\Dtil$ uses. The figure of merit below is $\E[N]\cdot\E[\Dtil^2]/B^2$, minimized at an intermediate $\kappa$ (about $11$ near $\kappa^\star\approx1.5$ for the logistic link, Table~\ref{tab:e1}). Since $\kappa=2L\rho/\tau$, the algorithm controls $\kappa$ through $\rho$ when $\tau$ is small but is confined to $\kappa\le L\delta/\tau$ when $\tau$ is large.
\end{remark}

\section{From gap estimates to Goldstein stationarity}
\label{sec:algorithms}

\paragraph{Gradient estimator.} For $x\in\R^d$, sample $u\sim\Unif(\sphere)$, run Definition~\ref{def:estimator} on the pair $(x-\rho u,x+\rho u)$ with fresh comparisons, and set
\begin{equation}
\label{eq:G}
G(x;u):=\frac{d}{2\rho}\,\Dtil(x-\rho u,x+\rho u)\,u .
\end{equation}

\begin{lemma}[Moments of $G$]
\label{lem:G}
Under Assumption~\ref{ass:lens}, $\E[G(x;u)\mid x]=\nabla f_\rho(x)$ and
\[
\E\big[\|G(x;u)\|^2\mid x\big]\;\le\;\sigma_G^2:=dL^2\Big(c_1+\frac{d\,\Clink}{\kappa^2}\Big),\qquad \kappa=\frac{2L\rho}{\tau}.
\]
\end{lemma}
\begin{proof}
Unbiasedness follows from Theorem~\ref{thm:estimator}(i), conditioning on $u$, and \eqref{eq:twopoint}. For the second moment, $\E[\Dtil^2\mid u]=\gap_u(x)^2+\Var(\Dtil\mid u)\le \gap_u(x)^2+\Clink\tau^2$ by Theorem~\ref{thm:estimator}(iii), so $\E\|G\|^2\le\frac{d^2}{4\rho^2}\big(\E_u[\gap_u(x)^2]+\Clink\tau^2\big)\le c_1dL^2+\frac{d^2\Clink\tau^2}{4\rho^2}$, and $\tau=2L\rho/\kappa$.
\end{proof}
The first term is the value-oracle term of \citet{kornowski2023algorithm}, linear in $d$ because $\gap_u$ concentrates over $u$. The second is comparison noise: $\Var(\Dtil\mid u)$ is of order $\tau^2$ regardless of $u$ and of $\rho$, and the factor $d^2/(4\rho^2)$ turns it into $d^2L^2\Clink/\kappa^2$, which dominates unless $\kappa^2\gtrsim d\,\Clink$; for the logistic link $\Clink/\kappa^2\le13(1+1/\kappa)^2$, so this term is at most $13d^2L^2(1+1/\kappa)^2$ for every $\kappa$. Averaging $k$ copies of $G$ per step reduces the bound to $\sigma_G^2/k+c_1dL^2$ at $k$ times the cost, which leaves the total complexities below unchanged in order as long as $k\lesssim\sigma_G^2/(c_1dL^2)=1+d\Clink/(c_1\kappa^2)$.

\paragraph{Main algorithm.} Algorithm~\ref{alg:o2nc} is the online-to-nonconvex conversion of \citet{cutkosky2023optimal} in the form used by \citet{kornowski2023algorithm}, with $G$ as the gradient oracle for $h:=f_\rho$. It maintains a clipped online-gradient-descent direction $\Delta_t$, evaluates the gradient estimator at a uniformly random point $z_t$ of each step, and outputs the average of $z_t$ over a random block of $\lfloor\nu/D\rfloor$ consecutive steps.

\begin{algorithm}[t]
\caption{Comparison-O2NC}
\label{alg:o2nc}
\begin{algorithmic}[1]
\STATE \textbf{Input:} $x_0$, smoothing radius $\rho$, radius $\nu$, clipping $D$, step $\eta$, horizon $T$.
\STATE $\Delta_1\gets0$.
\FOR{$t=1,\dots,T$}
\STATE $s_t\sim\Unif[0,1]$;\quad $x_t\gets x_{t-1}+\Delta_t$;\quad $z_t\gets x_{t-1}+s_t\Delta_t$.
\STATE $u_t\sim\Unif(\sphere)$; obtain $\Dtil_t\gets\Dtil(z_t-\rho u_t,z_t+\rho u_t)$ (Definition~\ref{def:estimator}, $N_t$ comparisons); $G_t\gets\frac{d}{2\rho}\Dtil_tu_t$.
\STATE $\Delta_{t+1}\gets\min\{1,\ D/\|\Delta_t-\eta G_t\|\}\,(\Delta_t-\eta G_t)$.
\ENDFOR
\STATE $M\gets\lfloor\nu/D\rfloor$, $K\gets\lfloor T/M\rfloor$, $\bar x_k\gets\frac1M\sum_{m=1}^{M}z_{(k-1)M+m}$ for $k\le K$; \textbf{output} $x^{\mathrm{out}}\gets\bar x_{k}$ with $k\sim\Unif\{1,\dots,K\}$.
\end{algorithmic}
\end{algorithm}

\begin{theorem}[Comparison complexity of Algorithm~\ref{alg:o2nc}]
\label{thm:main}
Let $f$ be $L$-Lipschitz and bounded below, $\delta,\eps\in(0,1)$, fix $\kappa^\star>0$, and set $\rho:=\min\{\delta/2,\ \kappa^\star\tau/(2L)\}$, so that $\kappa=2L\rho/\tau=\min\{\kappa^\star,L\delta/\tau\}$, and $\nu:=\delta-\rho\ge\delta/2$. Let the link satisfy Assumption~\ref{ass:lens} with $\alpha=\sigma(\kappa)$. Run Algorithm~\ref{alg:o2nc} with these $\rho,\nu$, $\eta=\Delta_h/(\sigma_G^2T)$, $D=\big(\nu^{1/2}\Delta_h/(\sigma_GT)\big)^{2/3}$, where $\Delta_h$ is any upper bound on $f_\rho(x_0)-\inf f_\rho$, for instance $\Df+L\rho$, or $\Df+H\rho^2$ when $f$ is $H$-smooth, and $\sigma_G$ is from Lemma~\ref{lem:G}. The algorithm uses $L$, $\Delta_h$, $\tau$ and the link as inputs, and $T$ is taken large enough that $D\le\nu$ (absorbed in $c$). There is an absolute constant $c$ such that for
\[
T\;=\;c\,\frac{\sigma_G^2\,\Delta_h}{\delta\,\eps^3}\;=\;O\!\left(d\Big(c_1+\frac{d\,\Clink}{\kappa^2}\Big)\frac{L^2(\Df+L\delta)}{\delta\eps^3}\right)
\]
the output satisfies $\E\gnorm{\delta}{f(x^{\mathrm{out}})}\le\eps$, and the expected number of comparisons is at most $T\,(1+\frac1{1-\alpha})$.
\end{theorem}
The proof, in Appendix~\ref{app:o2nc}, plugs Lemma~\ref{lem:G} into the guarantee of the online-to-nonconvex conversion for $h=f_\rho$ and uses $\partial_\nu f_\rho(x)\subseteq\partial_{\rho+\nu}f(x)=\partial_\delta f(x)$.

\paragraph{The simpler variant.} SGD on $f_\delta$ needs no clipping or blocks but pays for the $\sqrt d/\delta$ smoothness of the smoothed objective.

\begin{theorem}[Comparison-SGD]
\label{thm:sgd}
Under the assumptions of Theorem~\ref{thm:main} with $\rho=\delta$ (so $\kappa=2L\delta/\tau$) and $\eps\le L$, let $H_\delta:=c_2\sqrt dL/\delta$ and $\Delta_h:=\Df+L\delta$. The iteration $x_{t+1}=x_t-\eta G(x_t;u_t)$ with $\eta=\min\{1/H_\delta,\ \sqrt{2\Delta_h/(H_\delta\sigma_G^2T)}\}$ and output $x_R$, $R\sim\Unif\{0,\dots,T-1\}$, satisfies $\E\gnorm{\delta}{f(x_R)}\le\eps$ once
\[
T\;\ge\;\frac{32\,H_\delta\,\Delta_h\,\sigma_G^2}{\eps^4}\;=\;O\!\left(d^{3/2}\Big(c_1+\frac{d\,\Clink}{\kappa^2}\Big)\frac{L^3(\Df+L\delta)}{\delta\eps^4}\right),
\]
using at most $T(1+\frac1{1-\alpha})$ comparisons in expectation.
\end{theorem}
The proof (Appendix~\ref{app:sgd}) is the standard nonconvex SGD bound $\frac1T\sum_t\E\|\nabla f_\delta(x_t)\|^2\le\frac{2\Delta_h}{\eta T}+\eta H_\delta\sigma_G^2$, followed by $\nabla f_\delta(x)\in\partial_\delta f(x)$.

\paragraph{The smoothing radius is a free parameter.} The conversion uses only the Lipschitz constant of $f_\rho$, its initial gap and $\sigma_G^2$, not its smoothness, so $\rho$ can be taken small at no cost beyond $\rho+\nu\le\delta$, and it affects the rate only through $\kappa=2L\rho/\tau$. With $\rho=\min\{\delta/2,\kappa^\star\tau/(2L)\}$: when $\tau\le L\delta/\kappa^\star$, $\kappa=\kappa^\star$ and $\Clink$, $\E[N]$ are absolute constants of the link, so Theorem~\ref{thm:main} gives $O(d^2L^2\Df/(\delta\eps^3))$ comparisons uniformly in $\tau$; when $\tau>L\delta/\kappa^\star$, $\kappa=L\delta/\tau$ and the bound reads $O(d^2L^2\Df/(\kappa^2\delta\eps^3))$. Against $O(dL^2\Df/(\delta\eps^3))$ evaluations for a Lipschitz-per-sample value oracle, comparisons cost a factor $d/\kappa^2$, which the next section shows cannot be removed in the second regime. SGD pays for a small $\rho$ through the smoothness $c_2\sqrt dL/\rho$ of $f_\rho$, so Theorem~\ref{thm:sgd} is stated with $\rho=\delta$.

\section{A lower bound in the dimension and the temperature}
\label{sec:lower}

The lower bound uses only two properties of the oracle: each comparison is a single Bernoulli observation, and its success probability is a smooth function of the gap. Define the link constant
\[
\cs:=\sup_{t\in\R}\ \sup_{|s-t|\le1}\ \frac{\sigma'(s)^2}{\sigma(t)(1-\sigma(t))},
\]
which is at most $e^2/4$ for the logistic link and finite for the probit and cauchit links (Appendix~\ref{app:lower}). Taking $s=t$ shows $\cs\ge\sup_tI_\sigma(t)$ for the Fisher information of Remark~\ref{rem:kappa}, and for the logistic link $\cs\le e^2\sup_tI_\sigma(t)$.

\begin{theorem}[Lower bound]
\label{thm:lower}
Let $\sigma$ be a differentiable symmetric link with $\cs<\infty$, let $L,\delta,\eps,\tau>0$ with $\eps\le L/32$, set $\kappa_\eps:=\eps\delta/\tau$, and assume $d\ge2^{22}\kappa_\eps^2$. Consider any algorithm that adaptively issues at most $T$ comparison queries, possibly using internal randomness, and outputs a point $\hat x$. If for every convex, $L$-Lipschitz, $(\eps/\delta)$-smooth $f:\R^d\to\R$ with $f(0)-\inf f\le 128\,\eps\delta$ the output satisfies $\E\gnorm{\delta}{f(\hat x)}\le\eps$, then
\[
T\;\ge\;\frac{d^2\tau^2}{2^{25}\,\cs\,\eps^2\delta^2}\;=\;\frac{d^2}{2^{25}\,\cs\,\kappa_\eps^2}\;=\;\frac{d^2L^2}{2^{25}\,\cs\,\kappa_\delta^2\eps^2},\qquad \kappa_\delta:=\frac{L\delta}\tau .
\]
\end{theorem}

The bound is for a fixed budget of $T$ comparisons; it extends with a larger constant to a random number of comparisons with $\sup_v\E_v[N]\le T$, the budget model of Theorems~\ref{thm:main} and \ref{thm:sgd} (Appendix~\ref{app:lower}). On the larger, nonconvex class of $L$-Lipschitz $(\eps/\delta)$-smooth functions, an exact first-order oracle simulates a comparison with two queries, so the $\Omega(\Df H/\eps^2)$ bound of \citet{carmon2020lower}, in the globally Lipschitz form of \citet[Appendix~F]{cutkosky2023optimal}, gives $T=\Omega(\Df/(\delta\eps))$ in high dimension; this is the only dependence on $\Df$ we can prove.

\paragraph{Proof idea (details in Appendix~\ref{app:lower}).} Let $\phi$ be the Huber function with curvature $H:=\eps/\delta$ and slope $G:=32\eps$, and for $v\in\{-a,a\}^d$ with $a=16\delta/\sqrt d$ let $f_v(x):=\phi(\|x-v\|)$. Each $f_v$ is convex, $G$-Lipschitz (hence $L$-Lipschitz), $H$-smooth, with $\|\nabla f_v(x)\|=\min\{H\|x-v\|,G\}$; the slope $\Theta(\eps)$ rather than $L$ yields the factor $L^2/\eps^2$. Three steps.

\emph{Stationarity localizes $v$.} Since $\|\nabla f_v(\hat x)\|\le\gnorm{\delta}{f_v(\hat x)}+H\delta$, the guarantee $\E\gnorm{\delta}{f_v(\hat x)}\le\eps$ forces $\E\min\{\|\hat x-v\|,G/H\}\le2\delta$.

\emph{One comparison carries bounded information about $v$.} For any query $(x,y)$ and any coordinate $j$, let $v^{(j)}$ denote $v$ with the $j$th sign flipped. Because $v\mapsto f_v(z)$ is $L$-Lipschitz and $H$-smooth, $\sum_{j=1}^d\big(\gap_v(x,y)-\gap_{v^{(j)}}(x,y)\big)^2\le 64a^2L^2$, uniformly over the query. Since $\KL(\Bern(p)\|\Bern(p'))\le(p-p')^2/(p'(1-p'))$, one comparison contributes at most $\cs(\gap_v-\gap_{v^{(j)}})^2/\tau^2$ to the divergence between the $j$th pair of hypotheses, and the chain rule for adaptive queries gives $\sum_j\KL(P_v\|P_{v^{(j)}})\le 64\cs TG^2a^2/\tau^2$ after $T$ comparisons.

\emph{Assouad.} With $a^2:=d\tau^2/(128\cs TL^2)$ the average pairwise total variation is at most $1/2$, so any estimator of $v$ has mean squared error at least $\frac{R^2}2(1-s)$, where $R=a\sqrt d$ and $s^2=32\cs TG^2R^2/(d^2\tau^2)$. The clipped output $\Pi_{[-a,a]^d}(\hat x)$ has mean squared error at most $4R\delta$ by the first step; comparing the two bounds and optimizing $R$ gives the claimed bound. The cube radius $a\sqrt d=16\delta$ and the slope $G=32\eps$ optimize the resulting expression (Appendix~\ref{app:lower}); the side conditions ($4Ga\le\tau$ and $2a\sqrt d\le G/H$) are where $d\ge2^{22}\kappa_\eps^2$ and $\eps\le L/32$ enter, and $f_v(0)-\inf f_v=H\|v\|^2/2=128\eps\delta$.

\paragraph{Relation to existing lower bounds.} For the logistic link a comparison is $\mathbf 1\{f(y)+\Gamma'>f(x)+\Gamma\}$ with $\Gamma,\Gamma'$ i.i.d.\ Gumbel \citep{yellott1977relationship}, so Theorem~\ref{thm:lower} is the comparison-feedback analogue of the quadratic-in-$d$ bound of \citet{shamir2013complexity} for independently noised evaluations, with Goldstein stationarity as the target; Appendix~\ref{app:related-lb} relates it to \citet{duchi2015optimal,shamir2017optimal,jamieson2012query}. For $\tau\ge L\delta/\kappa^\star$, where Algorithm~\ref{alg:o2nc} uses $\rho=\delta/2$ and its $\kappa$ equals $\kappa_\delta$, the dominant term of Theorem~\ref{thm:main} and the bound of Theorem~\ref{thm:lower} read $O(\E[N]\Clink d^2\tau^2\Delta_h/(\delta^3\eps^3))$ and $\Omega(d^2\tau^2/(\cs\delta^2\eps^2))$: for a fixed link they agree in $d,\tau$ and $\delta/\eps$, and the remaining polynomial gap is $\Delta_h/(\delta\eps)$. On the hard instances themselves ($\Df=128\eps\delta$, $(\eps/\delta)$-smooth, so $\Delta_h=O(\eps\delta)$) the two agree up to a link-dependent factor for bounded $\kappa$, under the high-dimensional side condition. Appendix~\ref{app:lower} records why a chain of such instances does not close the gap for larger $\Df$, and why value oracles escape the bound.

\section{Experiments}
\label{sec:experiments}

We test the estimator, Lemma~\ref{lem:G}, the $d^2/T$ rate behind Theorem~\ref{thm:lower}, and the two algorithms on synthetic functions (Appendix~\ref{app:experiments}) and on two small tasks: preference-based policy search on a pendulum swing-up and $\ell_1$ regression with a ReLU model on the diabetes data. Comparisons are simulated from the true objective values through the logistic link, and step sizes are selected by grid search per method (Appendix~\ref{app:experiments}). In the synthetic experiments both estimators are unbiased at every $\kappa$; the run-length estimator has the smaller second moment and the early-stopped nested estimator the smaller cost and cost--second-moment product, both minimized near $\kappa^\star\approx1.5$ (Table~\ref{tab:e1}); $\E\|G\|^2$ grows as $d$ with values and as $d^2$ with comparisons (Figure~\ref{fig:secondmoment}); the maximum-likelihood estimate of $v$ on the instance of Theorem~\ref{thm:lower} has mean squared error matching $d^2\tau^2/(TH^2)$ within $10\%$ (Figure~\ref{fig:lowerbound}); and on a Chebyshev--Rosenbrock function Comparison-O2NC reaches a gradient norm of $0.87\pm0.06$ against $0.08\pm0.03$ for the value oracle with ten times fewer calls (Figure~\ref{fig:optimization}).

\begin{figure}[t]
\centering
\includegraphics[width=0.62\linewidth]{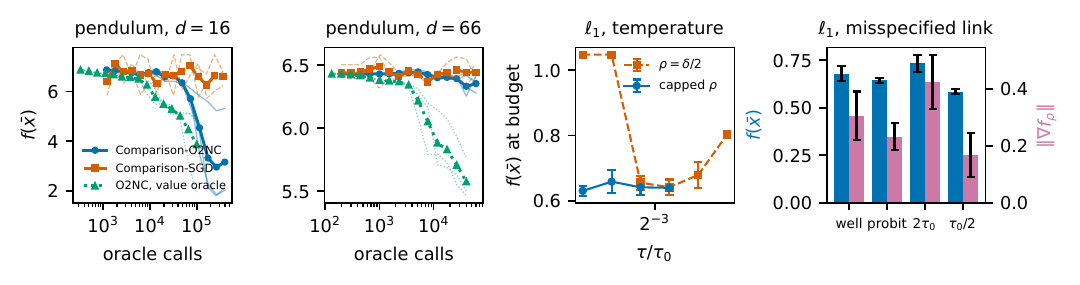}
\caption{Left: pendulum swing-up ($\delta=0.05$, $\kappa^\star=1.5$), cost of the output policy against oracle calls for $d=16$ ($1.2\cdot10^5$ iterations for comparisons, $5\cdot10^4$ for the value oracle) and $d=66$ ($2\cdot10^4$ iterations); thin lines are the 3 seeds. Right: $\ell_1$ regression, $d=49$, Comparison-O2NC with the step sizes selected at $\tau_0$, loss after $2\cdot10^5$ comparisons against the temperature with $\rho=\delta/2$ (nominal $\kappa$ from $384$ to $0.375$) and with the capped rule of Theorem~\ref{thm:main}; and loss and gradient norm under a misspecified link (matched-slope probit, true temperature $2\tau_0$ or $\tau_0/2$) while the estimator assumes the logistic link with $\tau_0$.}
\label{fig:pendulum}
\end{figure}

\paragraph{Pendulum (Figure~\ref{fig:pendulum}).} The objective is the average per-step cost $\theta^2+0.1\dot\theta^2+0.001u^2$ of a deterministic pendulum over 8 initial states and 100 steps, as a function of the parameters of a one-hidden-layer ReLU policy with clipped torque: nonsmooth, nonconvex, deterministic, with $L$ replaced by a finite-difference estimate of the local Lipschitz scale ($L\approx9$ for $d=16$; the theorems assume a global constant). Two policies are compared through the logistic link applied to their costs, as for trajectory preferences in preference-based reinforcement learning. With $d=16$, Comparison-O2NC swings the pendulum up from all initial states on two of three seeds ($4\cdot10^5$ comparisons; cost $6.9\to2.0$ and $2.1$, $5.3$ on the third), the value oracle on one of three ($10^5$ evaluations), and Comparison-SGD stays between $6.2$ and $6.8$. The comparison-based runs use about four times more oracle calls for a comparable outcome, far below the ratio of about $260$ between the second moments at the initial policy, because the queried gaps are only about $0.3\,\tau_0$; with $d=66$ neither oracle solves the task within $2\cdot10^4$ iterations.

\begin{figure}[t]
\centering
\includegraphics[width=0.62\linewidth]{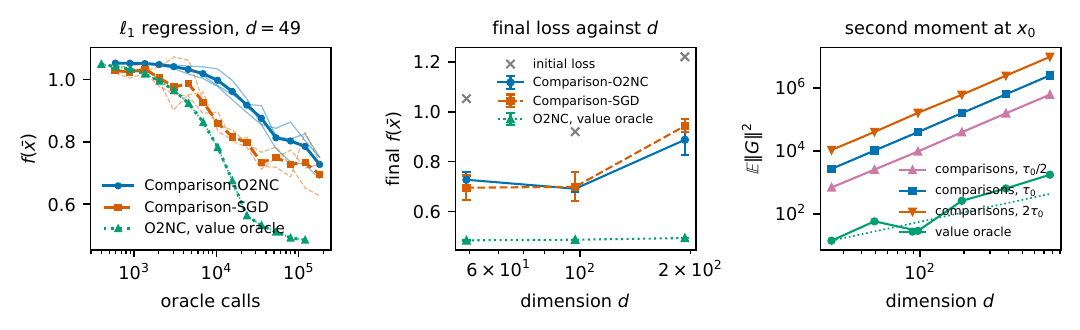}
\caption{$\ell_1$ regression with a ReLU model on the diabetes data ($\delta=0.05$, $\kappa^\star=1.5$, 3 seeds). Left: loss against oracle calls for $d=49$. Middle: final loss after $6\cdot10^4$ iterations against $d$. Right: $\E\|G\|^2$ at the initial point against $d$ with the temperature fixed across $d$ (dotted lines $\propto d$ and $\propto d^2$).}
\label{fig:l1reg}
\end{figure}

\paragraph{$\ell_1$ regression (Figure~\ref{fig:l1reg}).} The objective is the mean absolute error of a one-hidden-layer ReLU regressor on the standardized diabetes data ($442$ samples, $10$ features), a nonsmooth nonconvex function of $d=12w+1$ parameters for width $w$; cheap evaluations allow longer runs and a sweep over $d$. For $d=49$ Comparison-O2NC reduces the loss from $1.05$ to $0.73\pm0.03$ and Comparison-SGD to $0.70\pm0.05$ with $1.8\cdot10^5$ comparisons, against $0.484\pm0.001$ for the value oracle with $1.2\cdot10^5$ evaluations. At a fixed number of iterations the comparison-based methods stay above the value oracle for $d\in\{49,97,193\}$ (middle panel: $0.73$, $0.69$, $0.89$ against $0.48$ to $0.49$). The right panel isolates the cause: at fixed temperature $\E\|G\|^2$ for comparisons grows by a factor $951$ from $d=25$ to $769$, against $946$ for $d^2$, while the value oracle grows much more slowly. The nominal $\kappa=1.5$ is a worst case: the largest gap over $4000$ random directions is $0.1$ to $0.4$ of $\tau_0$, so in the linear regime of the link, where $\E[\Dtil^2]\approx2\tau^2$, a temperature smaller by that factor would reduce the $260$-fold second-moment ratio at $d=49$ by its square; the baseline study uses $\tau_0/32$.

\paragraph{Temperature and misspecification (Figure~\ref{fig:pendulum}, right).} At a budget of $2\cdot10^5$ comparisons, $\rho=\delta/2$ works while comparisons stay informative and collapses once they become nearly deterministic: at $\tau_0/64$ an estimate costs about $3\cdot10^4$ comparisons, the budget covers seven iterations, and the loss stays at its initial $1.05$, whereas the capped rule gives $0.63$ to $0.66$ for every $\tau\le\tau_0$. For $\tau>\tau_0$ the rules coincide and the loss degrades to $0.80$ at $4\tau_0$. Misspecifying the link shape or the temperature by a factor two moves the loss by at most $0.1$ from its well-specified value $0.68$ (Appendix~\ref{app:experiments}).

\paragraph{Baselines (Appendix~\ref{app:baselines}).} We compare against eight methods on the same oracle and budget, each tuned by grid search: sign-SGD on the two-point comparison \citep[the update of][]{chayti2025random,saha2021dueling}, direct search with votes \citep{bergou2020stochastic,bakkali2026noisy}, the one-bit direction estimator of \citet{cai2022one}, the ranking estimator of \citet{tang2024zeroth}, a biased plug-in in place of $\Dtil$, and O2NC with a value oracle. At $\tau=\tau_0/32$ the methods using a magnitude (ours, the plug-in, sign-SGD) reach $f\approx0.35$ and $\|\nabla f_\rho(\bar x)\|\in[0.23,0.39]$ on Chebyshev--Rosenbrock, the direct searches stall at $0.43$ with gradient norms above $1.1$, and the direction-recovery methods do not converge. On $\ell_1$ regression all reach $0.49$ to $0.55$, sign-SGD being best ($0.493$ against $0.503$) with a radius $\rho=2\delta$: a sign costs the same at large and small gaps, whereas the cost of $\Dtil$ grows as $e^{2L\rho/\tau}$. The plug-in is within $0.01$ of $\Dtil$ on both tasks (Table~\ref{tab:baselines}): at these budgets the experiments suggest that magnitude information matters on Chebyshev--Rosenbrock while sign methods win on $\ell_1$ regression, and show no practical advantage of unbiasedness, which is what the analysis requires.

\section{Related work}
\label{sec:related}

For convex objectives, \citet{jamieson2012query} treat a gap-dependent reliability, \citet{saha2021dueling,saha2025dueling} online convex optimization with a link function, \citet{cai2022one} one-bit recovery of sparse gradient directions. For smooth nonconvex objectives, \citet{bergou2020stochastic,elbakkali2025almost} analyze direct search with exact comparisons, \citet{tang2024zeroth} ranking oracles, \citet{lobanov2024acceleration} accelerated methods, \citet{bakkali2026noisy} tie-aware comparisons, \citet{chen2025compo} preference alignment, and \citet{chayti2025random} stochastic random search, which under $\delta'$-inexact comparisons converges up to a floor $O(\sqrt d\,\delta')$; these use smoothness to control the linearization error, while ours recovers the gap without bias, hence without a floor. For nonsmooth Lipschitz functions, $(\delta,\eps)$-Goldstein stationarity is the tractable target \citep{zhang2020complexity,kornowski2022oracle,jordan2023deterministic}; with value oracles \citet{kornowski2023algorithm} reach $O(d\delta^{-1}\eps^{-3})$ \citep[see also][]{lin2022gradient,chen2023faster,cutkosky2023optimal}, and lower bounds are per parameter \citep{duchi2015optimal,arjevani2023lower,carmon2020lower}, not joint.

\section{Conclusion and limitations}
\label{sec:conclusion}

We gave nonasymptotic guarantees for $(\delta,\eps)$-Goldstein stationarity from noisy comparisons: an $O(d(1+d\Clink/\kappa^2)\delta^{-1}\eps^{-3})$ upper bound and a lower bound $\Omega(d^2\tau^2/(\eps^2\delta^2))$ tailored to the comparison model, matching the upper bound in $d$ and $\tau$ on a small-initial-gap family and leaving the joint dependence on $\Df$ open. Limitations: the link must be known and lens-analytic; the lower bound is informative only when $\tau\gtrsim L\delta$; the constants of Theorem~\ref{thm:estimator} are loose; and sign-based updates, which carry no such guarantee, are empirically competitive at the budgets we could afford.

\subsection*{AI use statement}

We used a large language model as an assistant for restructuring and editing the manuscript, checking
and tightening proofs, drafting parts of the experiment code, and suggesting related work. The authors
verified all such output. The research ideas, the method and the main results are the authors' own,
and the authors take full responsibility for the content of this work.
\subsection*{Reproducibility statement}
All assumptions are stated in Sections~\ref{sec:prelim} and \ref{sec:estimator}; complete proofs are in Appendix~\ref{app:proofs}. The experiments use only NumPy and are fully specified in Appendix~\ref{app:experiments}; the experiments are fairly simple and the code will be made public.

\bibliography{refs}
\bibliographystyle{iclr2027_conference}

\appendix
\section{Proofs}
\label{app:proofs}

\subsection{Proof of Proposition~\ref{prop:expansion}}
\label{app:lens}

Let $U$ be an open neighborhood of $\lens_r$ on which $g$ is holomorphic. Since $r>1/2$, the circles $\{|z|=r\}$ and $\{|z-1|=r\}$ intersect at $\frac12\pm i\sqrt{r^2-1/4}$, and the boundary of $\lens_r$ consists of the arc $\Gamma_0\subset\{|z|=r\}$ with $|z-1|\le r$ and the arc $\Gamma_1\subset\{|z-1|=r\}$ with $|z|\le r$. For $z$ in the interior $\lens_r^\circ$, Cauchy's integral formula gives
\[
g(z)=\frac1{2\pi i}\oint_{\partial\lens_r}\frac{g(\zeta)}{\zeta-z}\,d\zeta=I_0(z)+I_1(z),\qquad I_j(z):=\frac1{2\pi i}\int_{\Gamma_j}\frac{g(\zeta)}{\zeta-z}\,d\zeta .
\]
The function $I_0$ is holomorphic on $\Cx\setminus\Gamma_0$, which contains the open disc $D_r:=\{|z|<r\}$, and $I_1$ is holomorphic on $\Cx\setminus\Gamma_1\supseteq 1-D_r$. Set $k(z):=I_1(1-z)$, holomorphic on $D_r$. The set $\lens_r^\circ$ is connected and invariant under $z\mapsto1-z$, and $g(z)=-g(1-z)$ holds on the real segment $(1-r,r)\subset\lens_r^\circ$, hence on all of $\lens_r^\circ$ by the identity theorem. For $z\in\lens_r^\circ$ we therefore have both $g(z)=I_0(z)+k(1-z)$ and $g(z)=-g(1-z)=-I_0(1-z)-k(z)$. Averaging the two identities,
\[
g(z)=h(z)-h(1-z),\qquad h:=\tfrac12(I_0-k),
\]
with $h$ holomorphic on $D_r$. Write $h(z)=\sum_{m\ge0}c_mz^m$ on $D_r$. For $r'<r$ the Cauchy estimates give $|c_m|\le C_{r'}(r')^{-m}$ with $C_{r'}=\max_{|z|=r'}|h(z)|$. For $p\in[1-\alpha,\alpha]$ we have $p,1-p\in[0,\alpha]\subset[0,r')$, so both $\sum_mc_mp^m$ and $\sum_mc_m(1-p)^m$ converge absolutely and uniformly on the interval, the constant terms cancel in the difference, and $g(p)=\sum_{m\ge1}c_m\big(p^m-(1-p)^m\big)$. \qed

The coefficients can be taken real: $g$ is real on the real segment of the lens, so by Schwarz reflection $g(\bar z)=\overline{g(z)}$ on the conjugation-symmetric set $\lens_r^\circ$, and $h_{\mathbb R}(z):=\frac12\big(h(z)+\overline{h(\bar z)}\big)$ is holomorphic on $D_r$ with real Taylor coefficients and satisfies $h_{\mathbb R}(z)-h_{\mathbb R}(1-z)=\frac12\big(g(z)+\overline{g(\bar z)}\big)=g(z)$, using $\overline{1-\bar z}=1-z$. We use $h_{\mathbb R}$ below, so the $c_m$ and the estimator are real.

A link with an analytic but not lens-analytic inverse: $g(p)=\operatorname{artanh}(2p-1)+\epsilon\,(p-\tfrac12)/((p-\tfrac12)^2+\eta^2)$ with $\epsilon/\eta^2$ small enough that $g'>0$ on $(0,1)$ and $\eta<\sqrt{\alpha^2-1/4}$; $g$ is odd about $1/2$, strictly increasing, and analytic on a neighborhood of $[1-\alpha,\alpha]$, but it has poles at $\tfrac12\pm i\eta$ inside the lens $\{|z|<\alpha\}\cap\{|z-1|<\alpha\}$, so no expansion with $\sum_m|c_m|\alpha^m<\infty$ exists.

\paragraph{The probit link.} Example~\ref{ex:links} asserts that $\Phi^{-1}$ continues analytically to any simply connected domain avoiding $0$ and $1$. Two facts are used. First, for an entire function $F$ let $S$ be the closure of the set of its critical values and of its finite asymptotic values (limits of $F$ along paths to $\infty$); then $F:\Cx\setminus F^{-1}(S)\to\Cx\setminus S$ is a covering map \citep[Section~2]{bergweiler1993iteration}, so every local inverse of $F$ continues analytically along every path in $\Cx\setminus S$, and by the monodromy theorem single-valuedly on every simply connected domain in $\Cx\setminus S$. Second, for $F=\Phi$ one has $S=\{0,1\}$. There are no critical values since $\Phi'=\varphi\ne0$. For the asymptotic values write $\Phi(z)=1-\frac12\operatorname{erfc}(w)$ with $w=z/\sqrt2$, and recall $\operatorname{erfc}(w)=\frac{e^{-w^2}}{w\sqrt\pi}(1+O(|w|^{-2}))$ as $|w|\to\infty$ in $|\arg w|\le\frac{3\pi}4-\epsilon$ \citep[\S7.12]{olver2010nist}, with $\operatorname{erfc}(-w)=2-\operatorname{erfc}(w)$ covering the remaining directions. Let $w(t)\to\infty$ be a path along which $\Phi$ converges to a finite value. On the tail of the path $|w|$ is large, so $\operatorname{Re}(w^2)=-|w|^2$ on the imaginary axis makes $|\Phi|\to\infty$ there; the tail therefore stays in one open half-plane, say $\operatorname{Re}w>0$, where $\operatorname{erfc}(w)=e^{-(w^2+\log w)}(1+o(1))/\sqrt\pi$. If $\operatorname{Re}(w^2)+\log|w|\to+\infty$ then $\operatorname{erfc}(w)\to0$ and $\Phi\to1$. Otherwise $\operatorname{Re}(w^2)+\log|w|$ is bounded above along a subsequence, and on the tail it is also bounded below, since $|\operatorname{erfc}(w)|$ is bounded; then $|w^2|=|w|^2\to\infty$ forces $|\operatorname{Im}(w^2)|\to\infty$, and $\operatorname{Im}(w^2)$ cannot change sign for large $|w|$ (at a sign change $|w|^2=|\operatorname{Re}(w^2)|=O(\log|w|)$), so the argument of $e^{-w^2}$ winds without bound while its modulus stays of order one, and $\operatorname{erfc}(w)$ has no limit, a contradiction. In the left half-plane the same argument gives $\Phi\to0$. Hence the finite asymptotic values are exactly $0$ and $1$.

For the logistic link, $h(z)=-\log(1-z)$ satisfies $h(p)-h(1-p)=\log\frac{p}{1-p}$, is holomorphic on $D_1$, and has $c_m=1/m$; the decomposition of the proof produces an $h$ that may differ from this one by a function $k_0$ with $k_0(z)=k_0(1-z)$, which changes the coefficients but not the identity $g=\sum_mc_mb_m$ on the operated range.

\subsection{Proof of Theorem~\ref{thm:estimator} and of Remark~\ref{rem:rr}}
\label{app:estimator}

Let $K:=N-1\ge1$ be the length of the initial run of equal outcomes, so that $A_m-B_m=(2Y_1-1)\mathbf 1\{m\le K\}$ and $\Dtil=\tau(2Y_1-1)S_K=\tau\sum_{m\ge1}c_m(A_m-B_m)$. Conditionally on $Y_1=1$, $K$ is geometric with $\Prob(K\ge k\mid Y_1=1)=p^{k-1}$; conditionally on $Y_1=0$, $\Prob(K\ge k\mid Y_1=0)=(1-p)^{k-1}$. In both cases the continuation probability $q\in\{p,1-p\}$ satisfies $q\le\alpha$.

\paragraph{(i) Unbiasedness.} $\E|c_m(A_m-B_m)|=|c_m|(p^m+(1-p)^m)\le2|c_m|\alpha^m\le2C(\alpha\rho_c)^m$, which is summable since $\alpha\rho_c<1$. By Fubini's theorem, $\E[\Dtil]=\tau\sum_{m\ge1}c_m\E[A_m-B_m]=\tau\sum_{m\ge1}c_mb_m(p)=\tau g(p)=\gap$.

\paragraph{(ii) Cost.} For $n\ge1$, $\{N>n\}$ is the event that the first $n$ outcomes are equal, so $\Prob(N>n)=p^n+(1-p)^n$, and $\Prob(N>0)=1$. Hence $\E[N]=\sum_{n\ge0}\Prob(N>n)=1+\sum_{n\ge1}(p^n+(1-p)^n)=1+\frac p{1-p}+\frac{1-p}{p}$. On $[1-\alpha,\alpha]$ this is maximized at the endpoints, where it equals $1+\frac\alpha{1-\alpha}+\frac{1-\alpha}\alpha\le1+\frac\alpha{1-\alpha}+1=1+\frac1{1-\alpha}$.

\paragraph{(iii) Second moment.} Since $\Prob(K=k)=p^k(1-p)+(1-p)^kp$ for $k\ge1$,
\begin{equation}
\label{eq:m2exact}
\E[\Dtil^2]=\tau^2\sum_{k\ge1}\big(p^k(1-p)+(1-p)^kp\big)S_k^2,
\end{equation}
which for the logistic link ($S_k=H_k$) is the exact identity of the theorem. In general $|S_k|\le\sum_{m\le k}|c_m|\le C\sum_{m\le k}\rho_c^m\le Ck\bar\rho^{\,k}$, and $p^k(1-p)+(1-p)^kp\le2\alpha^k$, so
\[
\E[\Dtil^2]\le2\tau^2C^2\sum_{k\ge1}k^2(\alpha\bar\rho^{\,2})^k=\frac{2C^2\tau^2\theta(1+\theta)}{(1-\theta)^3},\qquad\theta=\alpha\bar\rho^{\,2},
\]
using $\sum_{k\ge1}k^2\theta^k=\theta(1+\theta)/(1-\theta)^3$. For $\rho_c=1/r'$ with $r'\in(\sqrt\alpha,1)$ one has $\bar\rho=1/r'$ and $\theta=\alpha/r'^2<1$.

For the logistic link, $S_k=H_k\le1+\log k$, so $\E[\Dtil^2]\le\tau^2\E[(1+\log K)^2]\le\tau^2(2+2\E[\log^2K])$. We bound $\E[\log^2K]$ for a geometric $K$ with $\Prob(K\ge k)=q^{k-1}$, $q\le\alpha$. Let $k_0:=\lceil1/(1-\alpha)\rceil$. Since $\log\alpha\le-(1-\alpha)$, $\Prob(K>jk_0)=q^{jk_0}\le\alpha^{j/(1-\alpha)}\le e^{-j}$ for every $j\ge1$. Splitting $\{K\le k_0\}$ from the blocks $\{jk_0<K\le(j+1)k_0\}$ and using $\log^2((j+1)k_0)\le2\log^2(j+1)+2\log^2k_0$,
\[
\begin{aligned}
\E[\log^2K]\le\log^2k_0+\sum_{j\ge1}e^{-j}\big(2\log^2(j+1)+2\log^2k_0\big)\\
\qquad\le\Big(1+\frac2{e-1}\Big)\log^2k_0+1.1\le2.2\log^2k_0+1.1,
\end{aligned}
\]
where $2\sum_{j\ge1}e^{-j}\log^2(j+1)<1.1$ by direct summation. For the logistic link $1/(1-\alpha)=1+e^\kappa\le2e^\kappa$, so $k_0\le2/(1-\alpha)\le4e^\kappa$ and $\log k_0\le\kappa+\log4\le1.39(1+\kappa)$. Therefore $\E[\Dtil^2]\le\tau^2\big(2+2(2.2\cdot1.39^2(1+\kappa)^2+1.1)\big)\le\tau^2\big(4.2+8.6(1+\kappa)^2\big)\le13\,\tau^2(1+\kappa)^2$.

\paragraph{(iv) Moments.} $\E|\Dtil|^q=\tau^q\sum_k\Prob(K=k)|S_k|^q\le2\tau^qC^q\sum_kk^q(\alpha\bar\rho^{\,q})^k$, finite when $\alpha\bar\rho^{\,q}<1$. For the logistic link $|S_k|\le1+\log k$ and $\Prob(K=k)\le2\alpha^k$, so every moment is finite. \qed

\begin{remark}[Random truncation]
\label{rem:rr}
The randomly truncated series $\Dhat_{\mathrm{RR}}:=\tau\sum_{m\ge1}c_m(A_m-B_m)\mathbf 1\{M\ge m\}/\Prob(M\ge m)$, with $M$ an independent geometric variable of parameter $1-\beta$ (the Russian-roulette device of \citealp{mcleish2011general,rhee2015unbiased,lyne2015russian,beatson2019efficient}), is also unbiased, and $\E[\Dhat_{\mathrm{RR}}\mid Y_1,Y_2,\dots]=\Dtil$, so $\Var(\Dtil)\le\Var(\Dhat_{\mathrm{RR}})$ whenever the latter is finite. Its cost is $\min\{M,N\}$ comparisons, since the stream may be stopped at the first outcome change, so it is cheaper than $\Dtil$; the weights $1/\Prob(M\ge m)$ give it heavy tails. Appendix~\ref{app:rr} works this out; at the settings of Table~\ref{tab:e1} it uses about $70\%$ of the comparisons for a second moment about $10\%$ larger.
\end{remark}
\subsection{The Russian-roulette estimator}
\label{app:rr}

This section proves Remark~\ref{rem:rr} and records the properties of the randomly truncated alternative to $\Dtil$, which is the estimator a reader familiar with unbiased truncation \citep{mcleish2011general,rhee2015unbiased,lyne2015russian,beatson2019efficient} would write down first.

\paragraph{Definition.} Let $M\ge1$ be independent of the comparison stream with $q_m:=\Prob(M\ge m)>0$ for every $m$, and set
\[
\Dhat_{\mathrm{RR}}:=\tau\sum_{m\ge1}c_m\,(A_m-B_m)\,\frac{\mathbf 1\{M\ge m\}}{q_m}.
\]
The sum is finite: it has at most $\min\{M,N-1\}$ nonzero terms, since $A_m=B_m=0$ for $m\ge N$. We call $q_m=\beta^{m-1}$, that is $M$ geometric with parameter $1-\beta$, the geometric truncation, and we write $\alpha:=\sigma(\kappa)$ for the largest success probability in the operated range.

\paragraph{Unbiasedness and the Rao--Blackwell relation.} Conditionally on the stream $Y_1,Y_2,\dots$, each weight has $\E[\mathbf 1\{M\ge m\}]/q_m=1$, so
\begin{equation}
\label{eq:rb}
\E\big[\Dhat_{\mathrm{RR}}\mid Y_1,Y_2,\dots\big]=\tau\sum_{m\ge1}c_m(A_m-B_m)=\Dtil ,
\end{equation}
provided the series is absolutely convergent, which holds under Assumption~\ref{ass:lens} as in the proof of Theorem~\ref{thm:estimator}(i). Taking expectations gives $\E[\Dhat_{\mathrm{RR}}]=\gap$, and the law of total variance gives
\[
\Var(\Dhat_{\mathrm{RR}})=\Var(\Dtil)+\E\big[\Var(\Dhat_{\mathrm{RR}}\mid Y_1,Y_2,\dots)\big]\ \ge\ \Var(\Dtil)
\]
whenever the left-hand side is finite. So the run-length estimator is the Rao--Blackwellization of the truncated one with respect to the comparison stream, and can only reduce variance. It is not, however, uniformly better, because the two differ in cost.

\paragraph{Cost.} Computing $\Dhat_{\mathrm{RR}}$ requires the outcomes $Y_1,\dots,Y_m$ only for those $m$ with $M\ge m$ and $A_m-B_m\ne0$, so the stream may be stopped at the first outcome change or at the truncation time, whichever comes first:
\begin{equation}
\label{eq:rrcost}
N_{\mathrm{RR}}=\min\{M,\,N\},\qquad \E[N_{\mathrm{RR}}]\le\min\{\E[M],\E[N]\}=\min\Big\{\tfrac1{1-\beta},\,1+\tfrac p{1-p}+\tfrac{1-p}p\Big\}.
\end{equation}
Nested Russian roulette is therefore never more expensive than the run-length estimator and is typically cheaper, which is the reverse of the variance ordering. At $\beta=\sqrt\alpha$ and $p=\alpha$ the two effects nearly cancel: Table~\ref{tab:e1} reports a cost ratio of about $0.7$ and a second-moment ratio of about $1.1$, so the product that governs the complexity in Theorem~\ref{thm:main} favors the truncated estimator by $20$ to $30\%$ over the range $\kappa\in[0.25,3]$.

\paragraph{Second moment.} With one stream for all $m$ (the nested version) and the identity $(A_m-B_m)(A_n-B_n)=A_{m\vee n}+B_{m\vee n}$,
$$
\E[\Dhat_{\mathrm{RR}}^2]=\tau^2\sum_{m,n\ge1}\frac{c_mc_n}{q_{m\wedge n}}\big(p^{m\vee n}+(1-p)^{m\vee n}\big),
$$
since $\E[\mathbf 1\{M\ge m\}\mathbf 1\{M\ge n\}]/(q_mq_n)=q_{m\vee n}/(q_mq_n)=1/q_{m\wedge n}$. For the logistic link and geometric truncation this is $\tau^2\sum_{m,n\ge1}\frac{\beta^{1-(m\wedge n)}}{mn}(p^{m\vee n}+(1-p)^{m\vee n})$, which converges if and only if $\beta\ge\alpha$ at $p=\alpha$. Indeed, for $\beta<\alpha$ the diagonal terms diverge, whereas at $\beta=\alpha$, writing $n=m+k$, the summands are bounded by $2\alpha^{k+1}/m^2$, which is summable over $m\ge1$ and $k\ge0$. Drawing independent blocks of sizes $1,\dots,M$ instead of reusing one stream gives the same expectation and, for nonnegative coefficients such as the logistic $c_m=1/m$, a smaller second moment, at a cost of $M(M+1)/2$ comparisons; its product of cost and second moment is larger at every $\kappa$ we tested.

\paragraph{Tails.} The weights $1/q_m$ are what make the truncated estimator heavy-tailed. For the geometric truncation, on the event $\{M=m,\ A_m=1\}$, of probability $(1-\beta)\beta^{m-1}p^m$, one has $\Dhat_{\mathrm{RR}}\ge\tau\beta^{1-m}/m$, so
\[
\Prob\Big(\Dhat_{\mathrm{RR}}\ge\frac{\tau\beta^{1-m}}m\Big)\ \ge\ (1-\beta)(\beta p)^{m-1}p ,
\]
and $\E|\Dhat_{\mathrm{RR}}|^q=\infty$ as soon as $q>1+\log(1/\alpha)/\log(1/\beta)$. In particular only finitely many moments exist, and the number of them is governed by the same ratio that controls convergence of the second moment. By contrast $\Dtil=\tau(2Y_1-1)S_{N-1}$ with $N-1$ geometric, so its $q$th moment is finite whenever $\alpha\bar\rho^{\,q}<1$ (Theorem~\ref{thm:estimator}(iv)); since $\bar\rho=\rho_c>1$ for every lens-analytic link (the lens radius is below $1$), this covers only finitely many $q$ in general, while for the logistic link $|S_k|=H_k=O(\log k)$ and all moments are finite. This matters for high-probability guarantees and for any use of clipping, not for Theorem~\ref{thm:main}, which needs only a second moment.

\paragraph{Which to use.} Either estimator can be substituted into Lemma~\ref{lem:G} and hence into Theorems~\ref{thm:main} and \ref{thm:sgd}: the analysis uses only unbiasedness, a bound on the second moment, and the expected number of comparisons. We state the results for $\Dtil$ because it has no truncation parameter to tune, its second moment admits the closed form \eqref{eq:m2exact}, and, for the logistic link, all its moments are finite; a practitioner willing to tune $\beta$ and to accept heavy tails can use the truncated version and expect a constant-factor saving of the size reported in Table~\ref{tab:e1}.

\subsection{The online-to-nonconvex conversion}
\label{app:o2nc}

We use the guarantee of \citet{cutkosky2023optimal} in the form stated by \citet[Theorem~8]{kornowski2023algorithm}.

\begin{theorem}[\citealp{cutkosky2023optimal}]
\label{thm:o2nc}
Let $\nu,\eps\in(0,1)$ and let $h:\R^d\to\R$ be differentiable and $L$-Lipschitz with $h(x_0)-\inf h\le\Delta_h$. Suppose that at every query point the gradient oracle returns an unbiased estimate of $\nabla h$ with second moment at most $\sigma^2$, independently across calls. Then there is $T=O(\sigma^2\Delta_h/(\nu\eps^3))$ such that Algorithm~\ref{alg:o2nc} with $\eta=\Delta_h/(\sigma^2T)$ and $D=(\nu^{1/2}\Delta_h/(\sigma T))^{2/3}$ satisfies $z_{(k-1)M+m}\in\mathbb B(\bar x_k,\nu)$ for all $m\le M$, $k\le K$, and
\[
\E\Bigg[\frac1K\sum_{k=1}^{K}\Big\|\frac1M\sum_{m=1}^{M}\nabla h\big(z_{(k-1)M+m}\big)\Big\|\Bigg]\le\eps .
\]
In particular $\E\gnorm{\nu}{h(x^{\mathrm{out}})}\le\eps$ for $x^{\mathrm{out}}\sim\Unif\{\bar x_1,\dots,\bar x_K\}$.
\end{theorem}
\paragraph{Proof of Theorem~\ref{thm:main}.}
By Lemma~\ref{lem:G}, $G$ is an unbiased estimator of $\nabla f_\rho$ with second moment at most $\sigma_G^2$, and $f_\rho$ is $L$-Lipschitz with $f_\rho(x_0)-\inf f_\rho\le\Delta_h$ (for instance $\Delta_h=\Df+L\rho$, since $\|f_\rho-f\|_\infty\le L\rho$ and $\inf f_\rho\ge\inf f$; or $\Df+H\rho^2$ when $f$ is $H$-smooth). The guarantee of the online-to-nonconvex conversion \citep[Theorem~8]{kornowski2023algorithm}, restated as Theorem~\ref{thm:o2nc} in Appendix~\ref{app:o2nc}, gives $\E\gnorm{\nu}{f_\rho(x^{\mathrm{out}})}\le\eps$ after $T=O(\sigma_G^2\Delta_h/(\nu\eps^3))=O(\sigma_G^2\Delta_h/(\delta\eps^3))$ steps, using $\nu\ge\delta/2$. Since $\partial_\nu f_\rho(x)\subseteq\partial_{\rho+\nu}f(x)=\partial_\delta f(x)$, the same bound holds for $\gnorm{\delta}{f(x^{\mathrm{out}})}$. The expected number of comparisons is bounded as described below.
\qed

In Theorem~\ref{thm:main} the oracle is $G$ of \eqref{eq:G} with $h=f_\rho$, which is differentiable and $L$-Lipschitz; the randomness of each call consists of $u_t$ and the $N_t$ comparison outcomes, all independent across $t$ and of the past, so conditionally on the past the estimates are independent and unbiased, which is the oracle model of the conversion. We take $T$ large enough that $D\le\nu$, so that $M\ge1$ and $K\ge1$ in Algorithm~\ref{alg:o2nc}. Conditionally on $u_t$ the count $N_t$ has mean $1+\frac{p_t}{1-p_t}+\frac{1-p_t}{p_t}\le1+\frac1{1-\alpha}$ by Theorem~\ref{thm:estimator}(ii), and $T$ is deterministic, so $\E\sum_{t\le T}N_t\le T(1+\frac1{1-\alpha})$.

\subsection{Proof of Theorem~\ref{thm:sgd}}
\label{app:sgd}

Let $F:=f_\delta$, which is $H_\delta:=c_2\sqrt dL/\delta$-smooth and satisfies $F(x_0)-\inf F\le\Df+L\delta=:\Delta_h$ (because $F\le f+L\delta$ and $\inf F\ge\inf f$). For $\eta\le1/H_\delta$ and an unbiased gradient estimator with second moment at most $\sigma_G^2$, smoothness gives $\E[F(x_{t+1})\mid x_t]\le F(x_t)-\eta\|\nabla F(x_t)\|^2+\frac{\eta^2H_\delta}{2}\sigma_G^2$; summing over $t<T$ and dividing by $\eta T$,
\[
\frac1T\sum_{t=0}^{T-1}\E\|\nabla F(x_t)\|^2\le\frac{\Delta_h}{\eta T}+\frac{\eta H_\delta\sigma_G^2}{2}\le\frac{2\Delta_h}{\eta T}+\eta H_\delta\sigma_G^2 .
\]
With $\eta=\min\{1/H_\delta,\ \eta^\star\}$ and $\eta^\star:=\sqrt{2\Delta_h/(H_\delta\sigma_G^2T)}$ the right-hand side is at most $2\sqrt{2H_\delta\Delta_h\sigma_G^2/T}+2\Delta_hH_\delta/T$. For $R$ uniform on $\{0,\dots,T-1\}$, Jensen's inequality gives $\E\|\nabla F(x_R)\|\le\eps$ as soon as $T\ge32H_\delta\Delta_h\sigma_G^2/\eps^4$ and $T\ge4\Delta_hH_\delta/\eps^2$; the first condition implies the second because $\sigma_G^2\ge c_1dL^2\ge L^2\ge\eps^2$ for $\eps\le L$. Finally $\nabla f_\delta(x)\in\partial_\delta f(x)$ \citep{lin2022gradient}, so $\gnorm{\delta}{f(x_R)}\le\|\nabla f_\delta(x_R)\|$. The count of comparisons is as in Appendix~\ref{app:o2nc}. \qed

\subsection{Proof of Theorem~\ref{thm:lower}}
\label{app:lower}

\paragraph{Hard instances.} Let $H:=\eps/\delta$ and $G:=32\eps$, and let $\phi(r):=\frac H2r^2$ for $0\le r\le G/H=32\delta$ and $\phi(r):=Gr-\frac{G^2}{2H}$ for $r>G/H$ (a Huber function of slope $G$). For $v\in\R^d$ put $f_v(x):=\phi(\|x-v\|)$. Then $f_v$ is convex, $G$-Lipschitz, continuously differentiable, $H$-smooth, minimized at $v$, and $\|\nabla f_v(x)\|=\min\{H\|x-v\|,G\}$. Since $\eps\le L/32$ we have $G\le L$, so every $f_v$ belongs to the class of $L$-Lipschitz $H$-smooth convex functions. The parameter set is the cube $\{-a,a\}^d$, whose half-diagonal $R:=a\sqrt d$ we keep symbolic for the moment and fix at the end:
\[
a:=\frac R{\sqrt d},\qquad 2R\le\frac GH .
\]
Because $\|v\|=R\le G/H$, the initial gap is $f_v(0)-\inf f_v=\phi(R)=\frac H2R^2$.

Taking the slope $G$ proportional to $\eps$ rather than to $L$ is what produces the factor $L^2/\eps^2$ over the bound obtained with $G=L$: the sensitivity of a comparison to the hidden parameter is proportional to the slope, while the localization accuracy forced by Goldstein stationarity is not.

\paragraph{Step 1: stationarity localizes $v$.} For an $H$-smooth function, every $g\in\partial_\delta f(x)$ is a convex combination of gradients at points within distance $\delta$ of $x$, each within $H\delta$ of $\nabla f(x)$; hence $\|\nabla f(x)\|\le\gnorm{\delta}{f(x)}+H\delta$. Applied to $f_v$ at the output $\hat x$ and using $H\delta=\eps$,
\[
\min\{H\|\hat x-v\|,\,G\}\ \le\ \gnorm{\delta}{f_v(\hat x)}+\eps .
\]
Let $\vhat:=\Pi_{[-a,a]^d}(\hat x)$ be the coordinatewise projection onto the cube. Since $v$ lies in the cube, $\|\vhat-v\|\le\|\hat x-v\|$, and also $\|\vhat-v\|\le2R\le G/H$; consequently $H\|\vhat-v\|\le\min\{H\|\hat x-v\|,G\}$. The hypothesis $\E\gnorm{\delta}{f_v(\hat x)}\le\eps$ therefore gives $\E\|\vhat-v\|\le2\eps/H=2\delta$ and
\begin{equation}
\label{eq:upper}
\E\|\vhat-v\|^2\le2R\,\E\|\vhat-v\|\le 4R\delta .
\end{equation}

\paragraph{Step 2: information per comparison.} Fix a query $(x,y)$ and write $\psi_z(w):=\phi(\|z-w\|)$, so that $\gap_v(x,y)=\psi_y(v)-\psi_x(v)$. The map $w\mapsto\psi_z(w)$ is $G$-Lipschitz, continuously differentiable and $H$-smooth. For $j\le d$ let $v^{(j)}$ be $v$ with its $j$th coordinate negated, so $\|v-v^{(j)}\|=2a$. By the fundamental theorem of calculus and $H$-smoothness of $\psi_z$,
\[
\big|\psi_z(v^{(j)})-\psi_z(v)\big|=2a\Big|\int_0^1\partial_j\psi_z\big(v+t(v^{(j)}-v)\big)dt\Big|\le2a\big(|\partial_j\psi_z(v)|+2Ha\big),
\]
so that, using $(u+w)^2\le2u^2+2w^2$ and $\|\nabla\psi_z(v)\|\le G$,
\[
\sum_{j=1}^d\big(\psi_z(v^{(j)})-\psi_z(v)\big)^2\le8a^2\big(\|\nabla\psi_z(v)\|^2+4H^2a^2d\big)\le16a^2G^2,
\]
where the last step uses $2Ha\sqrt d=2HR\le G$. Since $\gap_v-\gap_{v^{(j)}}=[\psi_y(v)-\psi_y(v^{(j)})]-[\psi_x(v)-\psi_x(v^{(j)})]$,
\begin{equation}
\label{eq:sumsq}
\sum_{j=1}^d\big(\gap_v(x,y)-\gap_{v^{(j)}}(x,y)\big)^2\le64a^2G^2,\qquad \big|\gap_v(x,y)-\gap_{v^{(j)}}(x,y)\big|\le2G\|v-v^{(j)}\|=4Ga .
\end{equation}
Both bounds hold for every query pair, however far apart: they use only that $f_v$ depends on $v$ in a $G$-Lipschitz and $H$-smooth way, uniformly in the query point. The outcome of the query is $\Bern(\sigma(t))$ under $v$ and $\Bern(\sigma(t'))$ under $v^{(j)}$, with $t=\gap_v/\tau$ and $t'=\gap_{v^{(j)}}/\tau$. Using $\KL(\Bern(p)\|\Bern(p'))\le\chi^2(\Bern(p)\|\Bern(p'))=(p-p')^2/(p'(1-p'))$, the mean value theorem, and the condition
\begin{equation}
\label{eq:cond3}
4Ga\le\tau
\end{equation}
(which gives $|t-t'|\le1$), we obtain $\KL\big(\Bern(\sigma(t))\,\|\,\Bern(\sigma(t'))\big)\le\cs\,(t-t')^2=\cs(\gap_v-\gap_{v^{(j)}})^2/\tau^2$. Let $P_v$ denote the joint law of the $T$ outcomes and of the algorithm's internal randomness under instance $v$. Conditionally on the past, the query pair is a measurable function of the history, so the chain rule for the Kullback--Leibler divergence gives
\begin{equation}
\label{eq:kl}
\sum_{j=1}^d\KL(P_v\|P_{v^{(j)}})\le T\sup_{x,y}\sum_{j=1}^d\frac{\cs(\gap_v(x,y)-\gap_{v^{(j)}}(x,y))^2}{\tau^2}\le\frac{64\cs TG^2a^2}{\tau^2}.
\end{equation}

\paragraph{Step 3: Assouad.} Since $|\vhat_j|\le a$ and $v_j\in\{\pm a\}$, $(\vhat_j-v_j)^2\ge a^2\mathbf 1\{\operatorname{sign}(\vhat_j)\ne\operatorname{sign}(v_j)\}$ (with any fixed convention for $\vhat_j=0$), so Assouad's lemma in its standard form gives
\[
\begin{aligned}
\frac1{2^d}\sum_{v}\E_v\|\vhat-v\|^2\ &\ge\ \frac{a^2}{2}\sum_{j=1}^d\frac1{2^d}\sum_v\big(1-\TV(P_v,P_{v^{(j)}})\big)\\
&\ge\ \frac{da^2}{2}\Bigg(1-\sqrt{\frac1{2d}\,\frac1{2^d}\sum_v\sum_{j=1}^d\KL(P_v\|P_{v^{(j)}})}\Bigg),
\end{aligned}
\]
where the last step uses Pinsker's inequality $\TV\le\sqrt{\KL/2}$ and concavity of the square root. By \eqref{eq:kl} the square root is at most $s:=\sqrt{32\cs TG^2a^2/(d\tau^2)}=\sqrt{32\cs TG^2R^2/(d^2\tau^2)}$. Combining with \eqref{eq:upper} and $da^2/2=R^2/2$,
\[
\frac{R^2}{2}(1-s)\ \le\ 4R\delta\qquad\Longrightarrow\qquad s\ge1-\frac{8\delta}R\qquad\Longrightarrow\qquad T\ \ge\ \Big(1-\frac{8\delta}R\Big)^2\frac{d^2\tau^2}{32\,\cs\,G^2R^2}.
\]
It remains to choose $R$ and $G$. The right-hand side is proportional to $(1-8\delta/R)^2/R^2$, which is maximized at $R=16\delta$, and the constraint $2R\le G/H$ then forces $G\ge32\eps$; taking $G=32\eps$ (admissible since $\eps\le L/32$) and $R=16\delta$ gives $s\ge1/2$, $a=16\delta/\sqrt d$, $\phi(R)=128\,\eps\delta$, and
\[
T\ \ge\ \frac{d^2\tau^2}{4\cdot32\,\cs\,(32\eps)^2(16\delta)^2}\ =\ \frac{d^2\tau^2}{2^{25}\,\cs\,\eps^2\delta^2}.
\]
With these values \eqref{eq:cond3} reads $2048\,\eps\delta\le\tau\sqrt d$, that is $d\ge2^{22}\kappa_\eps^2$, and the initial gap is $128\,\eps\delta$, which are the hypotheses of the theorem. Enlarging the cube weakens the bound as $1/R^2$, so no choice of the free parameters of this family yields a bound that grows with $\Df$. \qed

\paragraph{Why value oracles escape the bound.} A comparison carries bounded Fisher information about $v$ because it is a Bernoulli variable with a smooth parameter; a same-seed two-point difference $F(x+\rho u;\xi)-F(x-\rho u;\xi)$ has no such bound, since its noise contracts with $\rho$. The content of Theorem~\ref{thm:lower} is $\E\|\hat v-v\|^2\ge d^2/\operatorname{tr}\mathcal I_T$ with $\operatorname{tr}\mathcal I_T\le4\cs TG^2/\tau^2$: one factor $d$ for the coordinates to learn, one for receiving one number per query.

\paragraph{Why the slope, and not the Lipschitz constant, controls the bound.} Steps 1 and 3 are unchanged if the slope $G$ is replaced by any $G'\in[2\eps,L]$: the localization accuracy $2\delta$ forced by $(\delta,\eps)$-Goldstein stationarity comes from the curvature $H=\eps/\delta$, not from the slope. Step 2, on the other hand, is proportional to $G'^2$, because a comparison between two points reveals the hidden parameter only through differences of function values, whose sensitivity to $v$ is the slope. Taking $G'$ as small as the argument allows, namely $\Theta(\eps)$, is therefore optimal, and yields a bound larger by $(L/\eps)^2$ than the same argument run with $G'=L$.

\paragraph{Finiteness of $\cs$ for the running links.} For the logistic link $\sigma'=\sigma(1-\sigma)$ and $(\log\sigma')'=1-2\sigma\in(-1,1)$, so $\sigma'(s)\le e\,\sigma'(t)$ for $|s-t|\le1$ and $\sigma'(s)^2/(\sigma(t)(1-\sigma(t)))=\sigma'(s)^2/\sigma'(t)\le e\,\sigma'(s)\le e/4$; hence $\cs\le e/4$ (numerically $\cs\approx0.38$). For the probit link, $\sigma'=\varphi$ and, for $t\ge1$, $1-\Phi(t)\ge\varphi(t)\,t/(1+t^2)$ (Mills' ratio) while $\varphi(s)^2\le\varphi(t)^2e^{2t+1}$ for $|s-t|\le1$, so the ratio is at most $\varphi(t)e^{2t+1}(1+t^2)/t\to0$ as $t\to\infty$; by symmetry and continuity $\cs<\infty$. For the cauchit link, $\sigma'(t)=1/(\pi(1+t^2))$ and $1-\sigma(t)\sim1/(\pi t)$, so the ratio is $O(t^{-3})$ as $|t|\to\infty$ and again $\cs<\infty$.

\paragraph{Random budgets.} Theorem~\ref{thm:lower} assumes at most $T$ comparisons almost surely. Suppose instead that the number of comparisons $N$ is random with $\sup_v\E_v[N]\le T$, as in Theorems~\ref{thm:main} and \ref{thm:sgd}. Run the algorithm but stop it after $32T$ comparisons, outputting $\hat x$ if it has finished and $0$ otherwise. By Markov's inequality the stopped output differs from $\hat x$ with probability at most $1/32$ under every $v$, and since $\|\vhat-v\|\le2R$ always, the bound \eqref{eq:upper} becomes $\E\|\vhat-v\|^2\le4R\delta+4R^2/32$, which at $R=16\delta$ gives $96\delta^2$ in place of $64\delta^2$; the Assouad step then yields $s\ge1/4$ and the conclusion $T\ge d^2\tau^2/(2^{32}\cs\eps^2\delta^2)$. The comparison of Theorems~\ref{thm:main} and \ref{thm:lower} in Section~\ref{sec:lower} is therefore legitimate up to this constant.

\paragraph{Tightness on the hard instances.} The instances used above have $\Df=128\eps\delta$ and are $(\eps/\delta)$-smooth, so the initial gap of the smoothed function obeys $f_\rho(x_0)-\inf f_\rho\le\Df+H\rho^2\le\Df+\eps\delta/4$ for $\rho\le\delta/2$. Theorem~\ref{thm:main} then uses $T=O(\sigma_G^2\Delta_h/(\delta\eps^3))$ iterations with $\sigma_G^2=O(\Clink d^2\tau^2/\delta^2)$ at $\rho=\delta/2$, hence $O\big(\E[N]\Clink\,d^2\tau^2/(\delta^2\eps^2)\big)$ comparisons, against the $\Omega\big(d^2\tau^2/(\cs\,\delta^2\eps^2)\big)$ of Theorem~\ref{thm:lower}. For $\kappa=\Theta(1)$, where $\Clink$ and $\E[N]$ are absolute constants, the two match up to a factor depending only on the link, provided $d\ge2^{22}\kappa_\eps^2$ (about $4\cdot10^3$ when $\eps=L/32$ and $\kappa_\delta=1$). On this family, for a fixed link and bounded $\kappa$, the algorithm therefore matches the lower bound in $d$, $\tau$, $\delta$ and $\eps$, and the remaining gap between Theorems~\ref{thm:main} and \ref{thm:lower} is the dependence on $\Delta_h$.

\paragraph{A second, complementary term.} A comparison is a function of the two queried values and of independent randomness, so an algorithm with an exact first-order oracle simulates a comparison query with two of its own queries. Any lower bound for randomized first-order algorithms therefore transfers, up to a factor $2$, to the comparison model. Applying this to the class of $H$-smooth functions with $H=\eps/\delta$, and using that a $(\delta,\eps)$-Goldstein stationary point of an $H$-smooth function is a $2\eps$-stationary point, the bound $\Omega(\Df H/\eps^2)$ of \citet{carmon2020lower} for finding $\eps$-stationary points of $H$-smooth functions gives
\begin{equation}
\label{eq:cdhs}
T=\Omega\Big(\frac{\Df}{\delta\eps}\Big),
\end{equation}
valid for randomized algorithms once $d$ is polynomially large in $\Df/(\delta\eps)$, as required by the random-rotation argument of \citet{carmon2020lower}. Combining with Theorem~\ref{thm:lower},
\[
T=\Omega\Big(\max\Big\{\frac{d^2\tau^2}{\cs\,\eps^2\delta^2},\ \frac{\Df}{\delta\eps}\Big\}\Big).
\]
The second term is the only $\Df$-dependence we can prove; it is dominated by the first unless $\Df\gtrsim d^2\tau^2/(\cs\,\eps\delta)$, and it carries no dependence on $\tau$ or $d$, since it holds even for an exact value oracle.

\paragraph{On a joint bound in $(d,\delta,\eps,\Df)$.} The bound above uses a single-stage instance and is therefore independent of $\Df$. The natural attempt at a joint bound replaces it by a zero-chain $F_U(x)=F(U^\top x)$ with $U\in\R^{d\times K}$ having orthonormal columns, scaled to be $H$-smooth with gradients of order $\eps$ before completion, which forces $K=\Theta(\Df/(\delta\eps))$ stages and a per-stage decrease $\Theta(\eps\delta)$, and charges the algorithm for localizing each hidden direction $u_{i+1}$. This does not reproduce the factor $K$, for a reason worth recording. Completing the chain requires $|\langle x,u_j\rangle|\gtrsim\delta$ for all $j\le K$ simultaneously, hence query points of norm $R=\Theta(\delta\sqrt K)$; the sensitivity of the observed gap to $u_{i+1}$ is then of order $\eps R$ rather than $\eps\delta$, so a single comparison carries $O(\cs\eps^2R^2/(d\tau^2))=O(\cs\kappa_\eps^2K/d)$ nats about $u_{i+1}$, with $\kappa_\eps:=\eps\delta/\tau$. At the same time the alignment needed to activate a stage is only $\langle x/\|x\|,u_{i+1}\rangle\ge\Omega(1/\sqrt K)$, which costs $O(d/K)$ nats. The two effects give $T_i=\Omega(d^2/(\kappa_\eps^2K^2))$ per stage and $T=\Omega(d^2/(\kappa_\eps^2K))$ in total, which is decreasing in $K$ and is maximized at $K=1$, that is, at the single-stage instance used above. A joint bound therefore requires a hard family in which the sensitivity to the hidden parameter does not grow with the norm of the query, as it does not for the Huber instances, while retaining strict sequentiality across $\Theta(\Df/(\delta\eps))$ stages. We leave this open.

\subsection{Relation to existing lower bounds}
\label{app:related-lb}
 For the logistic link there is a reduction to independently noised function values: if $\Gamma,\Gamma'$ are i.i.d.\ Gumbel with scale $\tau$, then $\mathbf 1\{f(y)+\Gamma'>f(x)+\Gamma\}$ is Bernoulli with parameter $\sigma((f(y)-f(x))/\tau)$ \citep{bradley1952rank,yellott1977relationship}, so a comparison algorithm is an algorithm with independent per-evaluation noise, and lower bounds for that model transfer whenever their proof accommodates Gumbel noise. \citet{shamir2013complexity} shows that with independent noise the number of evaluations must scale at least quadratically with $d$ even for strongly convex smooth functions, whereas two evaluations under the same noise realization cost only a factor $\sqrt d$ \citep{duchi2015optimal,shamir2017optimal}; Theorem~\ref{thm:lower} is the comparison-feedback analogue, with the link's Fisher information and $\kappa$ explicit and with Goldstein stationarity as the target. The comparison-oracle lower bound of \citet{jamieson2012query} concerns strongly convex functions with reliability $\Prob(\text{correct})\ge\frac12+\min\{\delta_0,\mu|\gap|^{\kappa_J-1}\}$; a smooth link has $\kappa_J=2$ and $\mu\asymp1/\tau$, and their bound $\E[f(\hat x)-\inf f]\gtrsim\sqrt{d/T}$ converts, for $\mu_{\mathrm{sc}}$-strongly convex $f$, through $\|\nabla f\|^2\ge2\mu_{\mathrm{sc}}(f-\inf f)$ into a bound of order $d\,\mu_{\mathrm{sc}}^2/\eps^4$ comparisons for the mean-squared criterion $\E\|\nabla f(\hat x)\|^2\le\eps^2$, up to link constants (the expected-norm criterion of this paper does not follow without further moment control, so we use the comparison only qualitatively). That bound is linear in $d$, which is the dependence produced by needle-in-a-haystack arguments, where each query is informative about one of $d$ candidate locations, so the dimension counts places to search and the accuracy enters through the divergence. In the family of Theorem~\ref{thm:lower}, by contrast, the sensitivity of a query to the hidden parameter is spread across all coordinates while its total is bounded, which is where the second factor of $d$ comes from. The upper bound's term $d^2L^2\Df/(\kappa^2\delta\eps^3)$ combines both effects, and a matching joint bound would need one hard family exhibiting both; the natural candidate is the chain construction of \citet{arjevani2023lower} in the globally Lipschitz form of \citet[Appendix~F]{cutkosky2023optimal}, with the variance-throttling step replaced by the Bernoulli divergence bound above and the chain embedded in a hypercube of shifts. Even for value oracles no joint zero-order lower bound in $(d,\delta,\eps)$ is known \citep{kornowski2023algorithm}; we leave the comparison version as a conjecture.

\section{Experimental details and additional results}
\label{app:experiments}

All experiments run on a single CPU core with NumPy; the synthetic experiments take under thirty minutes in total and the two real tasks about forty minutes. Random seeds are fixed in the code. Throughout we use the logistic link and the run-length estimator of Definition~\ref{def:estimator}. Where a Lipschitz constant is needed it is estimated as the 95th percentile of finite-difference gradient norms over 100 random points around the initial point (the maximum is dominated by a few near-kink points and is reported alongside).

\begin{figure}[h]
\centering
\includegraphics[width=0.9\linewidth]{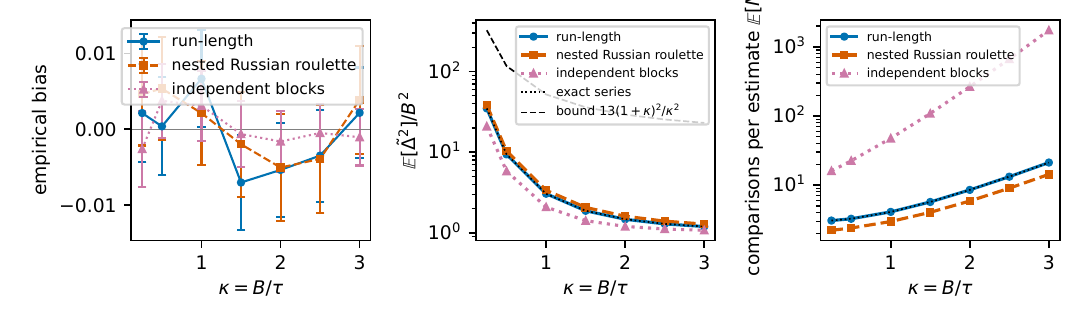}
\caption{Gap estimators for the logistic link at the edge of the operated range ($\gap=B$, $\tau=1$, $2\cdot10^5$ samples per point): the run-length estimator of Definition~\ref{def:estimator} and two Russian-roulette variants with geometric truncation $\beta=\sqrt\alpha$ (nested single stream, and independent blocks of sizes $1,\dots,M$). Left: empirical bias with two standard errors. Middle: second moment relative to $B^2$, with the exact series \eqref{eq:m2exact} and the bound $13(1+\kappa)^2/\kappa^2$. Right: comparisons per estimate; the dotted line is $1+\alpha/(1-\alpha)+(1-\alpha)/\alpha$.}
\label{fig:estimator}
\end{figure}

\paragraph{Estimator (Figure~\ref{fig:estimator}, Table~\ref{tab:e1}).} For each $\kappa\in\{0.25,0.5,1,1.5,2,2.5,3\}$ we set $\tau=1$, $B=\kappa$ and $\gap=B$ (the edge of the operated range, where the second moment and the cost are largest) and draw $2\cdot10^5$ independent estimates with the run-length estimator and, for comparison, with two Russian-roulette estimators with geometric truncation $\beta=\sqrt\alpha$: the nested version that reuses one stream for all $m$ (cost $\min\{M,N\}$, since the terms vanish after the first outcome change) and the version with independent blocks of sizes $1,\dots,M$ (cost $M(M+1)/2$). All three are unbiased in simulation (sample biases statistically consistent with zero: all within two standard errors, one point at $2.2$). The run-length estimator has a smaller second moment than the nested Russian-roulette estimator at every tested $\kappa$; the nested Russian-roulette estimator, stopped at the first outcome change as \eqref{eq:rrcost} allows, uses about $70\%$ as many comparisons for a second moment about $10\%$ larger, so its product of cost and second moment is smaller by $20$ to $30\%$. The independent-block version is far more expensive than both. The simulated second moments of $\Dtil$ agree with the exact series \eqref{eq:m2exact} to three digits, and the bound $13\tau^2(1+\kappa)^2$ is loose by a factor $15$ to $20$. Both products are minimized near $\kappa=1.5$ on this grid.

\begin{table}[h]
\centering\footnotesize
\caption{Estimator statistics at $\gap=B$ (logistic link, $\tau=1$). RR: Russian roulette with $\beta=\sqrt\alpha$.}
\label{tab:e1}
\resizebox{\linewidth}{!}{\begin{tabular}{rrrrrrrrrrr}
\toprule
& & \multicolumn{4}{c}{run-length} & \multicolumn{3}{c}{nested RR} & \multicolumn{2}{c}{independent RR}\\
$\kappa$ & $\alpha$ & $\E[\Dtil^2]/B^2$ & exact \eqref{eq:m2exact} & $\E[N]$ & $\E[N]\E[\Dtil^2]/B^2$ & $\E[\Dhat^2]/B^2$ & $\E[N]$ & product & $\E[\Dhat^2]/B^2$ & $\E[N]$\\
\midrule
0.25 & 0.562 & 34.95 & 34.93 & 3.06 & 107.1 & 38.29 & 2.21 & 84.8 & 20.94 & 16.0\\
0.50 & 0.622 & 9.41 & 9.43 & 3.25 & 30.6 & 10.36 & 2.39 & 24.8 & 5.83 & 22.5\\
1.00 & 0.731 & 3.06 & 3.06 & 4.09 & 12.5 & 3.38 & 2.97 & 10.0 & 2.09 & 47.4\\
1.50 & 0.818 & 1.88 & 1.88 & 5.69 & 10.7 & 2.07 & 4.03 & 8.3 & 1.42 & 108.6\\
2.00 & 0.881 & 1.48 & 1.48 & 8.51 & 12.6 & 1.61 & 5.88 & 9.5 & 1.20 & 265.5\\
2.50 & 0.924 & 1.29 & 1.30 & 13.23 & 17.1 & 1.40 & 9.01 & 12.6 & 1.12 & 667.4\\
3.00 & 0.953 & 1.20 & 1.20 & 21.17 & 25.4 & 1.29 & 14.29 & 18.4 & 1.08 & 1739.0\\
\bottomrule
\end{tabular}}
\end{table}

\begin{figure}[h]
\centering
\begin{minipage}[h]{0.48\linewidth}\centering
\includegraphics[width=0.9\linewidth]{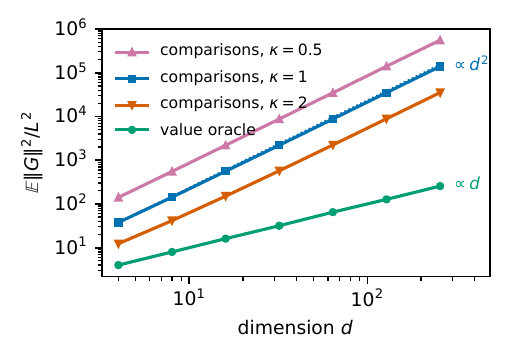}
\caption{Second moment of the gradient estimator $G$ for $f(x)=\|x\|_\infty$ ($\rho=0.1$, $4\cdot10^4$ directions per point). Values give $\E\|G\|^2\propto d$; comparisons give $\propto d^2$, with the constant decreasing in $\kappa$.}
\label{fig:secondmoment}
\end{minipage}\hfill
\begin{minipage}[h]{0.48\linewidth}\centering
\includegraphics[width=0.9\linewidth]{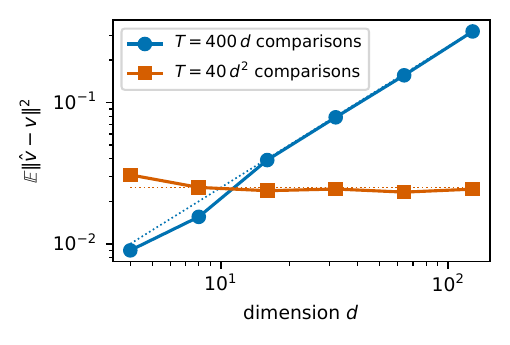}
\caption{Mean squared error of the logistic maximum-likelihood estimate of $v$ from $T$ comparisons of $\pm u$ on the quadratic instance ($H=\tau=1$, 30 repetitions). Dotted lines are the Fisher-information prediction $d^2\tau^2/(T H^2)$.}
\label{fig:lowerbound}
\end{minipage}
\end{figure}

\paragraph{Second moment and estimation rate (Figures~\ref{fig:secondmoment} and \ref{fig:lowerbound}).} For $f(x)=\|x\|_\infty$ the exact two-point estimator has $\E\|G\|^2\approx dL^2$ while the comparison-based estimator grows as $d^2$, which is the mechanism of Lemma~\ref{lem:G}: comparison noise does not benefit from the concentration of $\gap_u$. On the quadratic instance of Theorem~\ref{thm:lower} the maximum-likelihood estimate of $v$ from $T$ comparisons has mean squared error matching $d^2\tau^2/(TH^2)$ within $10\%$ for $d$ from $4$ to $128$: doubling $d$ requires quadrupling $T$.

\paragraph{Figure~\ref{fig:secondmoment}.} $f(x)=\|x\|_\infty$ ($L=1$), $\rho=0.1$, at a Gaussian random point whose first coordinate is shifted by $3$ so that the maximum is attained at a single coordinate. For each $d\in\{4,\dots,256\}$ we draw $4\cdot10^4$ directions $u$, compute the exact gaps $\gap_u$, and form $\frac{d^2}{4\rho^2}\gap_u^2$ (value oracle) and $\frac{d^2}{4\rho^2}\Dhat^2$ with one comparison-based estimate per direction at $\kappa\in\{0.5,1,2\}$, $\tau=B/\kappa$.

\paragraph{Figure~\ref{fig:lowerbound}.} The instance is $f_v(x)=\frac H2\|x-v\|^2$ (the quadratic region of the Huber function) with $H=1$, $v\in\{\pm0.05\}^d$, $\tau=1$, and queries $(-u,u)$ with $u\sim\Unif(\sphere)$, for which $\gap_v=-2H\langle u,v\rangle$. We fit $v$ by logistic regression (Newton's method, ridge $10^{-8}$) on $T$ outcomes and report the mean squared error over 30 repetitions, for $T=400d$ and $T=40d^2$. The prediction $d^2\tau^2/(TH^2)$ is the Cram\'er--Rao value with $\sigma'(0)=1/4$ and $\E[uu^\top]=I/d$.

\begin{figure}[h]
\centering
\includegraphics[width=0.72\linewidth]{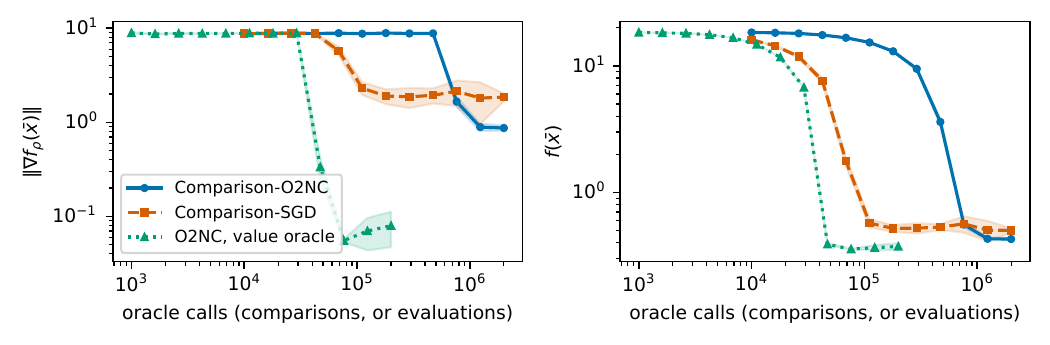}
\caption{Nonsmooth Chebyshev--Rosenbrock function in $d=10$ ($\delta=0.2$, $\rho=\nu=0.1$, $\kappa=1$). Comparison-O2NC and Comparison-SGD with a budget of $2\cdot10^6$ comparisons and O2NC with an exact two-point value oracle with $2\cdot10^5$ evaluations; step sizes tuned per method. Mean and standard deviation over 5 seeds; $\|\nabla f_\rho(\bar x)\|$ is estimated by Monte Carlo at block averages.}
\label{fig:optimization}
\end{figure}

\paragraph{Synthetic optimization (Figure~\ref{fig:optimization}).} $f(x)=\frac14|x_1-1|+\sum_{i=1}^{d-1}|x_{i+1}-2|x_i|+1|$ with $d=10$, $x_0=-\mathbf 1$ ($f(x_0)=18.5$), $\delta=0.2$, $\rho=\nu=0.1$, $L\approx8.8$, $\kappa=1$. Comparison-based methods use $2\cdot10^6$ comparisons, the value-oracle reference $2\cdot10^5$ evaluations. Step sizes were selected on one seed at half the budget by the final estimated $\|\nabla f_\rho\|$ from the grids $\eta\in\{10^{-6},3\cdot10^{-6},10^{-5},3\cdot10^{-5}\}\times D\in\{10^{-4},3\cdot10^{-4},10^{-3}\}$ (Comparison-O2NC), $\eta\in\{10^{-5},3\cdot10^{-5},10^{-4},3\cdot10^{-4},10^{-3}\}$ (Comparison-SGD), and $\eta\in\{10^{-5},3\cdot10^{-5},10^{-4}\}\times D\in\{10^{-4},3\cdot10^{-4},10^{-3}\}$ (value-oracle O2NC); selected: $(\eta,D)=(10^{-6},10^{-4})$, $\eta=10^{-5}$, and $(10^{-4},3\cdot10^{-4})$; the two comparison-based selections lie at the lower edge of their grids. Comparison-O2NC drives $f$ from $18.5$ to $0.43\pm0.01$ and $\|\nabla f_\rho(\bar x)\|$ to $0.87\pm0.06$ in $6.6\cdot10^5$ iterations ($\E[N]\approx3$ comparisons per iteration, since the queried gaps are well inside the operated range), Comparison-SGD to $0.50\pm0.02$ and $1.85\pm0.18$; the value-oracle reference reaches $0.37\pm0.02$ and $0.08\pm0.03$ with ten times fewer oracle calls. The gradient norm $\|\nabla f_\rho(\bar x)\|$ is estimated with $3000$ exact two-point samples (subtracting the sampling variance from the squared norm) at the averages $\bar x$ of the last eight blocks of $\lfloor\nu/D\rfloor$ consecutive points $z_t$ for O2NC, and at the last eight iterates for SGD.

\paragraph{Real tasks: common protocol.} For each task and dimension, the temperature is $\tau_0=L\delta/\kappa^\star$ with $\kappa^\star=1.5$ and $\delta=0.05$, so that $\rho=\delta/2$ gives $\kappa=\kappa^\star$; $\nu=\delta-\rho$. Step sizes are chosen by grid search on one seed at half the final iteration budget, selecting by the final objective value: for Comparison-O2NC, $\eta\in\eta_{\mathrm{th}}\cdot\{10,30,\dots,3\cdot10^3\}$ and $D\in D_{\mathrm{th}}\cdot\{3,10,30\}$, where $\eta_{\mathrm{th}}=\Delta_h/(\sigma_G^2T)$ and $D_{\mathrm{th}}=(\nu^{1/2}\Delta_h/(\sigma_GT))^{2/3}$ are the values of Theorem~\ref{thm:main} with $\sigma_G^2$ computed from Lemma~\ref{lem:G} using the empirical $\E[\Dtil^2]/B^2=1.88$ at $\kappa=1.5$; for Comparison-SGD $\eta\in\eta_{\mathrm{th}}\cdot\{3\cdot10^2,\dots,3\cdot10^5\}$; for the value oracle $\eta\in\eta_{\mathrm{th}}\cdot\{10^2,\dots,10^4\}$ rescaled by the ratio of the two second-moment bounds. Final runs use 3 seeds. All methods use the same number of iterations, so the comparison-based methods use $\E[N]$ comparisons per iteration against $2$ evaluations; since the queried gaps are small relative to $\tau_0$, $\E[N]\approx3$ in practice rather than the nominal $5.7$. Selected values: pendulum $d=16$: Comparison-O2NC $(\eta,D)=(1.1\cdot10^{-7},4.7\cdot10^{-4})$, Comparison-SGD $\eta=3.3\cdot10^{-3}$, value-oracle O2NC $(\eta,D)=(1.9\cdot10^{-6},1.4\cdot10^{-3})$; pendulum $d=66$: Comparison-O2NC $(\eta,D)=(1.1\cdot10^{-8},1.4\cdot10^{-4})$, Comparison-SGD $\eta=3.6\cdot10^{-5}$, value-oracle O2NC $(\eta,D)=(4.4\cdot10^{-7},4.3\cdot10^{-4})$; $\ell_1$ regression $d=49$: Comparison-O2NC $(\eta,D)=(6.9\cdot10^{-8},2.0\cdot10^{-4})$, Comparison-SGD $\eta=2.1\cdot10^{-4}$, value-oracle O2NC $(\eta,D)=(2.3\cdot10^{-5},2.0\cdot10^{-4})$; $d=97$: Comparison-O2NC $(\eta,D)=(8.1\cdot10^{-10},3.1\cdot10^{-5})$, Comparison-SGD $\eta=2.4\cdot10^{-5}$, value-oracle O2NC $(\eta,D)=(1.4\cdot10^{-6},9.4\cdot10^{-5})$; $d=193$: Comparison-O2NC $(\eta,D)=(1.8\cdot10^{-9},6.3\cdot10^{-5})$, Comparison-SGD $\eta=5.5\cdot10^{-6}$, value-oracle O2NC $(\eta,D)=(1.8\cdot10^{-7},6.3\cdot10^{-5})$. The extended pendulum runs and the temperature sweep start from a different random initialization than the corresponding grid search, with the step sizes selected there.

\paragraph{Pendulum.} Dynamics are those of the standard pendulum benchmark ($g=10$, $m=l=1$, $dt=0.05$, angular velocity clipped to $[-8,8]$, torque clipped to $[-2,2]$), with 8 fixed initial states drawn once (angle uniform in $[-\pi,\pi]$, velocity uniform in $[-1,1]$) and a horizon of 100 steps. The policy maps $(\cos\theta,\sin\theta,\dot\theta/8)$ through a ReLU layer of width $w$ to a scalar torque, $d=5w+1$. Widths $3$ and $13$ give $d=16$ and $d=66$; the initial policy has cost $7.4$ for $d=16$ and $6.4$ for $d=66$ in the grid search, and $6.9$ for $d=16$ in the extended runs, which start from a different initialization. For $d=16$ the final runs were extended, with the step sizes selected at $2\cdot10^4$ iterations, to $5\cdot10^4$ iterations for the value oracle and $1.2\cdot10^5$ iterations for the two comparison-based methods.

\paragraph{Temperature sweep and misspecified links ($\ell_1$ regression, $d=49$).} For $\tau\in\tau_0\cdot\{1/256,1/64,1/16,1/4,1,4\}$, Comparison-O2NC is run with the step sizes selected at $\tau_0$ until a budget of $2\cdot10^5$ comparisons is exhausted (at most $10^5$ iterations), either with $\rho=\delta/2$ (so the nominal $\kappa=L\delta/\tau$ ranges over $384,96,24,6,1.5,0.375$) or with $\rho=\min\{\delta/2,\kappa^\star\tau/(2L)\}$, which keeps $\kappa=1.5$ for $\tau\le\tau_0$. The number of comparisons per estimate depends on the actual success probabilities: since typical gaps are well below $B$, it is about $3$ for all capped runs and for $\rho=\delta/2$ down to $\tau_0/4$, about $4$ at $\tau_0/16$, about $3\cdot10^4$ at $\tau_0/64$ and $10^5$ at $\tau_0/256$, where the budget covers seven and two iterations. For the misspecified links the outcomes are generated by $\Phi(\gap/(1.702\tau_0))$ (a probit with the slope of the logistic at zero), by the logistic link with temperature $2\tau_0$, and by the logistic link with temperature $\tau_0/2$, while the estimator uses the logistic link with $\tau_0$; the estimator is then biased, by a factor close to $\tau_0/\tau_{\mathrm{true}}$ near zero gap in the last two cases.

\paragraph{$\ell_1$ regression.} Inputs and target of the diabetes data are standardized; the model is a ReLU layer of width $w$ followed by a linear output, $d=12w+1$; the loss is the mean absolute error over the 442 samples, whose initial value is $1.05$ for the width-$4$ initialization used in Section~\ref{sec:experiments} ($0.92$ and $1.22$ for widths $8$ and $16$; a least-squares linear fit has mean absolute error $0.56$). Widths $4,8,16$ ($d=49,97,193$) are run for $6\cdot10^4$ iterations with grid search at $3\cdot10^4$; the middle panel of Figure~\ref{fig:l1reg} reports, per method, the median over seeds of the number of oracle calls at the first checkpoint whose loss is at most $0.65$. The second-moment panel uses widths $2$ to $64$ ($d=25$ to $769$) at the initial point with $4000$ directions, $\rho=\delta/2$, and a temperature fixed across $d$ at $\tau_0=L\delta/\kappa^\star$ with $L$ the estimate for width $4$ (and $\tau_0/2$, $2\tau_0$). With the temperature fixed, the comparison term of Lemma~\ref{lem:G} is exactly $\propto d^2$; the value term grows faster than $d$ here because the gradient norm at a random initialization grows with the width. The largest observed gap over $4000$ directions is far below $B=2L\rho$ (between $0.14$ and $0.37$ of $\tau_0$), a consequence of the concentration of $\gap_u$: the nominal $\kappa$ is a worst case, and typical comparisons are made at gaps well inside the operated range.

\section{Baselines}
\label{app:baselines}

\begin{figure}[h]
\centering
\includegraphics[width=\linewidth]{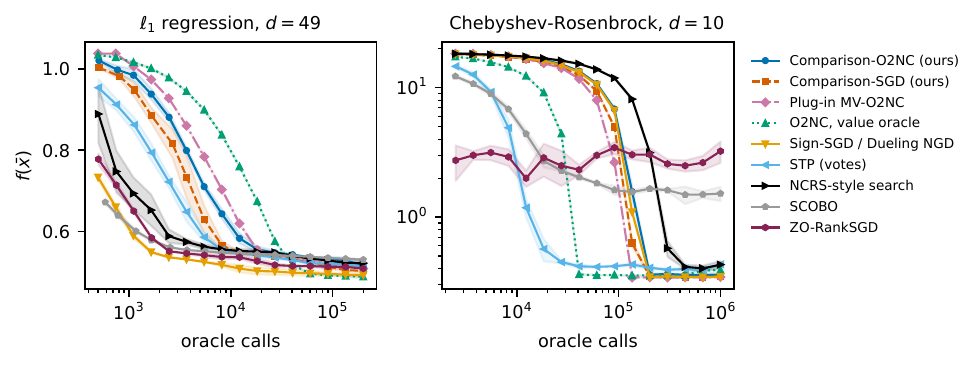}
\caption{Budget-matched comparison at $\tau=\tau_0/32$ on $\ell_1$ regression ($d=49$, $2\cdot10^5$ comparisons) and the nonsmooth Chebyshev--Rosenbrock function ($d=10$, $10^6$ comparisons). Mean and standard deviation over 3 seeds; hyperparameters selected per method by grid search at half budget.}
\label{fig:baselines}
\end{figure}

All methods use the same logistic comparison oracle and are charged one unit per comparison (the value oracle, two evaluations per step, is charged two). The objectives are deterministic, so a pair may be compared repeatedly; that is how the vote-based methods spend their budget.

\begin{itemize}[leftmargin=*,itemsep=0pt,topsep=2pt]
\item \textbf{Sign-SGD / dueling.} $x\gets x-\eta\,s\,u$, where $s\in\{\pm1\}$ is the majority outcome of $k$ comparisons of $x\pm\rho u$ (the direction in which $f$ is reported to increase); $k=1$ is the normalized-gradient update of \citet{saha2021dueling} and the stochastic random search of \citet{chayti2025random}, whose analysis of the sign of a noisy two-point difference also explains why per-comparison signal and noise scale alike, which is what we observe in Section~\ref{sec:experiments}. Grid: $\eta\in\{3\cdot10^{-4},\dots,3\cdot10^{-1}\}$, $\rho$ over the radii below and $\{\delta/2,2\delta\}$, $k\in\{1,5,21\}$.
\item \textbf{STP} \citep{bergou2020stochastic}: move to the better of $x\pm\alpha u$ when a majority of $k$ comparisons of that candidate against $x$ reports it smaller; each candidate is compared with $x$ separately, so an iteration costs $2k$ comparisons. \textbf{NCRS-style search} \citep{bakkali2026noisy}: one candidate $x+\alpha u$, accepted on a majority of $k$ comparisons. Grids: $\alpha$ as $\eta$ above, $k\in\{1,5,21\}$.
\item \textbf{SCOBO} \citep{cai2022one}: $m$ comparisons of $x$ against $x+rz_i$, $z_i$ Gaussian, decoded by the Plan--Vershynin program $\max\langle v,g\rangle$ over $\|g\|\le1$, $\|g\|_1\le\sqrt s$, then a normalized step with $\eta/\sqrt t$. Grid: $m\in\{d/2,2d\}$, $s\in\{d/4,d\}$, $k\in\{1,5\}$, $r=\delta$. We do not use its comparison-based line search.
\item \textbf{ZO-RankSGD} \citep{tang2024zeroth}: $m$ perturbed points ranked by their win counts over all $m(m-1)/2$ pairs ($k$ votes each), then the rank-weighted direction. Grid: $m\in\{4,8\}$, $k\in\{1,5\}$, $\mu=\delta$.
\item \textbf{Plug-in} (biased): $\tau\,\log(\hat p/(1-\hat p))$ with $\hat p$ the fraction of $k$ comparisons clipped to $[1/2k,1-1/2k]$, inside O2NC, $k\in\{5,21\}$.
\item \textbf{Ours and the value oracle:} Algorithm~\ref{alg:o2nc} and its SGD variant, with the smoothing radius from the capped rule $\rho=\min\{\delta/2,\kappa^\star\tau/(2L)\}$ for $\kappa^\star\in\{1.5,4,8,16,32,64,128\}$, $\eta$ over ten decades and $D\in\delta\cdot\{0.002,0.01,0.05\}$.
\end{itemize}

Selection is by the final objective on one seed at half budget; final runs use 3 different seeds at full budget, so the selection seed is not among the reported runs. Three seeds and grids whose optimum sometimes lies on an edge (noted where it occurs) limit the resolution of the ranking; differences below about $0.02$ in Table~\ref{tab:baselines} should not be read as significant. Table~\ref{tab:baselines} reports the final objective and the smoothed gradient norm $\|\nabla f_\rho(\bar x)\|$, estimated with $2000$ exact two-point samples at $\rho=\delta/2$ for every method, so that the stationarity measure is comparable across methods that use different internal radii.

\begin{table}[h]
\centering\footnotesize
\caption{Final objective and smoothed gradient norm at $\tau=\tau_0/32$ (mean $\pm$ std over 3 seeds).}
\label{tab:baselines}
\begin{tabular}{lrrrr}
\toprule
& \multicolumn{2}{c}{$\ell_1$ regression, $d=49$} & \multicolumn{2}{c}{Chebyshev--Rosenbrock, $d=10$}\\
method & $f(\bar x)$ & $\|\nabla f_\rho(\bar x)\|$ & $f(\bar x)$ & $\|\nabla f_\rho(\bar x)\|$\\
\midrule
Comparison-O2NC (ours) & $0.518\pm0.010$ & $0.078\pm0.022$ & $0.359\pm0.006$ & $0.387\pm0.035$\\
Comparison-SGD (ours) & $0.503\pm0.003$ & $0.105\pm0.024$ & $0.348\pm0.001$ & $0.297\pm0.020$\\
Plug-in majority vote & $0.509\pm0.008$ & $0.087\pm0.036$ & $0.344\pm0.004$ & $0.230\pm0.023$\\
Sign-SGD / dueling & $0.493\pm0.006$ & $0.062\pm0.007$ & $0.346\pm0.004$ & $0.254\pm0.035$\\
STP (votes) & $0.513\pm0.004$ & $0.144\pm0.064$ & $0.431\pm0.009$ & $1.132\pm0.157$\\
NCRS-style search & $0.522\pm0.010$ & $0.090\pm0.014$ & $0.429\pm0.028$ & $1.246\pm0.357$\\
SCOBO & $0.531\pm0.002$ & $0.073\pm0.020$ & $1.521\pm0.194$ & $7.029\pm0.380$\\
ZO-RankSGD & $0.509\pm0.000$ & $0.117\pm0.015$ & $3.231\pm0.608$ & $6.666\pm0.284$\\
\midrule
O2NC, value oracle & $0.489\pm0.003$ & $0.018\pm0.003$ & $0.390\pm0.023$ & $0.110\pm0.038$\\
\bottomrule
\end{tabular}
\end{table}

At $\tau_0$ and $\tau_0/8$ the ordering on $\ell_1$ regression is the same, with all comparison-based methods within $0.51$ to $0.56$; at $\tau_0$, where the queried gaps are far below the temperature, the sign of one comparison is already an unbiased estimate of the gap up to a scalar and the sign-based methods are ahead of ours throughout. Our selected radii sit at the largest value the grid allows on $\ell_1$ regression ($\rho=\delta/2$ for O2NC at $\tau_0/8$, $\delta/3$ at $\tau_0/32$), so on that problem the binding constraint is $\rho\le\delta/2$ rather than the cost of the estimator; sign-SGD selects $\rho=2\delta$, outside the range the Goldstein target allows for our smoothing argument.

\end{document}